\documentclass[11pt, a4paper, reqno]{amsart}
\newcommand{\N}{\mathbb{N}}
\newcommand{\A}{\mathbb{A}}
\newcommand{\Tr}{\operatorname{Tr}}
\usepackage[utf8]{inputenc}
\usepackage{amscd}
\usepackage{amssymb}
\usepackage{newunicodechar}
\usepackage[utf8]{inputenc}
\usepackage{amsthm}
\usepackage{listings}
\usepackage[colorlinks=true,urlcolor=black]{hyperref}  
\usepackage{tikz}
\usetikzlibrary{calc}
\newcommand{\CC}{\mathbb{C}}

\newcommand{\FF}{\mathbb{F}}

\newcommand{\al}{\alpha}
\newcommand{\be}{\beta}
\newcommand{\g}{\gamma} 
\newcommand{\ka}{\kappa}

\usepackage{amsmath}
\usepackage{graphicx,enumitem}
\setlist[enumerate]{topsep=0pt, itemsep=0pt, parsep=0pt, partopsep=0pt}
\usepackage{latexsym}

\usepackage{bm}
\usepackage{upref}
\usepackage{color}
\usepackage{xcolor}
\usepackage{subcaption}
\usepackage[utf8]{inputenc}
\usepackage[T1]{fontenc}
\usepackage{adjustbox}
\theoremstyle{plain}
\newtheorem{theorem}{Theorem}[section]
\newtheorem{prop}[theorem]{Proposition}
\newtheorem{lemma}[theorem]{Lemma}
\newtheorem{corollary}[theorem]{Corollary}

\theoremstyle{definition}
\newtheorem{defi}[theorem]{Definition}
\newtheorem{ex}[theorem]{Example}

\theoremstyle{remark}
\newtheorem{remark}[theorem]{Remark}

\theoremstyle{plain}

\numberwithin{equation}{section}
\allowdisplaybreaks
\setlist{leftmargin=15pt,labelindent=15pt}
\setlist[enumerate]{wide=0pt, leftmargin=15pt, labelwidth=15pt, align=left}
\setlist[enumerate]{label=(\roman*),topsep=0pt,itemsep=2pt,parsep=0pt}
\newcommand{\Fq}{\mathbb{F}_q}

\newcommand{\F}{\mathbb{F}}

\setlist[enumerate]{label=(\roman*)}

\makeatletter
\renewcommand{\subsection}{\@startsection{subsection}{2}{\z@}%
  {-3.25ex \@plus -1ex \@minus -.2ex}%
  {1.5ex \@plus .2ex}%
  {\normalfont\bfseries}}
\makeatother

\title[Rational Points on Danielewski Surfaces and Polygonal Numbers]{Rational Points on Danielewski Surfaces over Finite fields and Polygonal Numbers}

\author[Sakshi Gupta]{Sakshi Gupta}
\author[Anit Kuckian]{Anit Kuckian}
\author[Indranath Sengupta]{Indranath Sengupta$^{*}$}
\address[] {\newline
Indian Institute of Technology Gandhinagar, 
    Palaj, Gandhinagar, Gujarat 382355, India}
\email[] {sakshi.gupta2693@gmail.com, anit.kuckian@iitgn.ac.in, indranathsg@iitgn.ac.in}
\thanks{$^{*}$Corresponding Author}
\keywords{Danielewski surfaces, double Danielewski surfaces, finite fields, rational points, character sums, K\"onig--Rados theorem, Polygonal numbers.}  
\subjclass[2010]{11T06, 11T99}

\begin{document}

\vspace{-9mm}
\begin{abstract}
We study the number of $\mathbb{F}_q$-rational points on Danielewski and double Danielewski surfaces over the finite field $\mathbb{F}_q$ and show an interesting connection with the polygonal numbers. 

\end{abstract}

\maketitle

\section{Introduction}\label{s1}

\noindent
Let $\Fq$ denote the finite field with $q$ elements and let $\Fq^\times=\Fq\setminus\{0\}$. Counting rational points on algebraic varieties over finite fields is an important problem in finite field theory and arithmetic geometry. Exact formulas are often difficult to obtain. Thus, one studies classes of affine varieties for which exact point counts can be obtained from the defining equations. In recent years, several authors have obtained exact formulas of $\Fq$-rational points for certain affine varieties over finite fields, especially for diagonal, triangular, and related systems of equations, using exponent matrices, Smith normal form, linear congruences, character sums, Jacobi sums, and Gauss sums; see, for example, \cite{ChenCao2023,HongZhu2025,HuQinZhao2019,ZhuQiangLi2023}.

\medskip
\noindent
The present paper follows this general direction for a different class of affine varieties, namely Danielewski surfaces and double Danielewski surfaces. Danielewski surfaces are closely related to the cancellation problem in affine algebraic geometry \cite{Danielewski}.
For a field $k$, the cancellation problem asks whether, for finitely generated $k$-algebras $A$ and $B$ and an indeterminate $U$, a $k$-algebra isomorphism
$$
A[U]\cong_k B[U]
$$
implies
$$
A\cong_k B.
$$

\medskip
\noindent
In 1989, Danielewski constructed counterexamples to the cancellation problem in dimension two over $\CC$ \cite{Danielewski}. For a polynomial $P(Z)\in\CC[Z]$ of degree at least $2$ with distinct roots and $d \in \mathbb{N}$, he considered the affine surface
$$
S_d:\ X^dY-P(Z)=0.
$$
If
$$
B_d=\CC[X,Y,Z]/(X^dY-P(Z))
$$
denotes the coordinate ring of $S_d$, then for $d_1\neq d_2$,
$$
B_{d_1}\not\cong_{\CC} B_{d_2}
\qquad\text{but}\qquad
B_{d_1}[U]\cong_{\CC} B_{d_2}[U].
$$
Thus, the rings $B_d$ give counterexamples to the cancellation problem in dimension two. Related results on the cancellation problem can be found in \cite{BhatwadekarGupta, GuptaSurvey,GuptaInvent,GuptaAdv}.

\medskip 
\noindent 
Motivated by Danielewski's construction, Gupta and Sen introduced the class of double Danielewski surfaces \cite{Gupta2019}. Over a field $k$, such a surface is defined in the affine $4$-space $\A^4_k$ with coordinates $X,Y,Z,T$ by the system 
$$ X^{d_1}Y-P(X,Z)=0, \qquad X^{d_2}T-Q(X,Y,Z)=0, $$
where $d_1,d_2\in\N$, $P(X,Z)\in k[X,Z]$ is monic in $Z$, $Q(X,Y,Z)\in k[X,Y,Z]$ is monic in $Y$, $\deg_Z P(X,Z)\geq 2$, and $\deg_Y Q(X,Y,Z)\geq 2$. If 
$$ B_{d_1,d_2} = k[X,Y,Z,T]/ \bigl(X^{d_1}Y-P(X,Z),\,X^{d_2}T-Q(X,Y,Z)\bigr) $$
denotes the corresponding coordinate ring, then under suitable hypotheses one has 
$$ B_{d_1,d_2}\not\cong_{k} B_{d_1,d_2+1} \qquad\text{but}\qquad B_{d_1,d_2}[U]\cong_{k} B_{d_1,d_2+1}[U]. $$
In this paper, we study Danielewski and double Danielewski surfaces over $\Fq$ and determine the number of their $\Fq$-rational points.

\medskip
\noindent
A further motivation comes from the appearance of polygonal-number patterns in the number of $\Fq$-rational points on certain double Danielewski surfaces. For suitable choices of the defining polynomials $P$ and $Q$, these point counts form subsequences of the square, hexagonal, octagonal, and decagonal number sequences as $q$ varies over prime powers.
Figures~\ref{fig:square-numbers} and \ref{fig:hexagonal-numbers} illustrate the first few square and hexagonal numbers. 

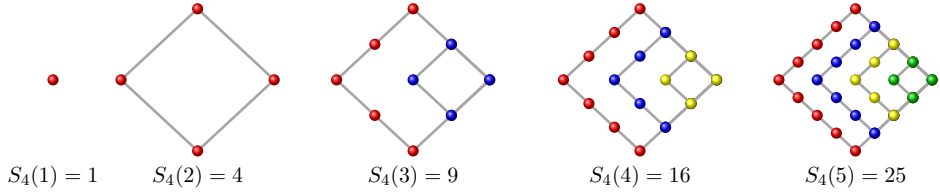
\begin{figure}[h]
\centering
\begin{tikzpicture}[scale=0.75, transform shape]
\colorlet{sline}{gray!70}
\def\dotsize{5.6pt}

\newcommand{\ballpts}[2]{%
  \node[
    circle,
    shade,
    ball color=#1,
    inner sep=0pt,
    minimum size=\dotsize
  ] at #2 {};
}

\newcommand{\SquareNest}[3]{%
  \begin{scope}[shift={(#1,0)}]
    \ifnum#2=1
      \ballpts{red}{(0,0)}
      \node at (0,-1.7) {#3};
    \else
      \coordinate (L0) at (0,0);
      \coordinate (T0) at (1.35,1.25);
      \coordinate (R0) at (2.70,0);
      \coordinate (B0) at (1.35,-1.25);

      \pgfmathtruncatemacro{\Nsq}{#2}

      \foreach \m in {\Nsq,...,2} {
        \pgfmathsetmacro{\s}{(\m-1)/(\Nsq-1)}
        \pgfmathtruncatemacro{\depth}{\Nsq-\m+1}

        \def\shellcolor{red}
        \ifnum\depth=2\relax \def\shellcolor{blue}\fi
        \ifnum\depth=3\relax \def\shellcolor{yellow!100!black}\fi
        \ifnum\depth=4\relax \def\shellcolor{green!80!black}\fi
        \ifnum\depth=5\relax \def\shellcolor{orange}\fi

        \coordinate (L) at ($(R0)!\s!(L0)$);
        \coordinate (T) at ($(R0)!\s!(T0)$);
        \coordinate (R) at (R0);
        \coordinate (B) at ($(R0)!\s!(B0)$);

        \draw[sline, line width=1pt] (L)--(T)--(R)--(B)--cycle;

        \pgfmathtruncatemacro{\mm}{\m-2}
        \foreach \P/\Q in {L/T,T/R,R/B,B/L} {
          \foreach \k in {0,...,\mm} {
            \pgfmathsetmacro{\t}{\k/(\m-1)}
            \ballpts{\shellcolor}{($(\P)!\t!(\Q)$)}
          }
        }
      }

      \node at (1.35,-1.7) {#3};
    \fi
  \end{scope}
}

\SquareNest{0.0}{1}{$S_4(1)=1$}
\SquareNest{1.2}{2}{$S_4(2)=4$}
\SquareNest{5.0}{3}{$S_4(3)=9$}
\SquareNest{9.0}{4}{$S_4(4)=16$}
\SquareNest{12.8}{5}{$S_4(5)=25$}

\end{tikzpicture}
\caption{The first five square numbers.}
\label{fig:square-numbers}
\end{figure}

\begin{figure}[h]
\centering
\begin{tikzpicture}[scale=0.73, transform shape]

\colorlet{hexline}{gray!70}
\def\dotsize{5.6pt}

\newcommand{\ballpt}[2]{%
  \node[
    circle,
    shade,
    ball color=#1,
    inner sep=0pt,
    minimum size=\dotsize
  ] at #2 {};
}

\newcommand{\HexShell}[5]{%
  \coordinate (#1A) at ($(#2)+(0:#3)$);
  \coordinate (#1B) at ($(#2)+(60:#3)$);
  \coordinate (#1C) at ($(#2)+(120:#3)$);
  \coordinate (#1D) at ($(#2)+(180:#3)$);
  \coordinate (#1E) at ($(#2)+(240:#3)$);
  \coordinate (#1F) at ($(#2)+(300:#3)$);

  \draw[hexline, line width=1pt]
    (#1A)--(#1B)--(#1C)--(#1D)--(#1E)--(#1F)--cycle;

  \pgfmathtruncatemacro{\mminusone}{#4-1}
  \foreach \P/\Q in {A/B,B/C,C/D,D/E,E/F,F/A}{
    \foreach \k in {0,...,\mminusone}{
      \pgfmathsetmacro{\t}{\k/#4}
      \ballpt{#5}{($(#1\P)!\t!(#1\Q)$)}
    }
  }
}

\ballpt{red}{(0,0)}
\node at (0,-2.0) {$S_6(1)=1$};

\coordinate (O2) at (2.1,0);
\HexShell{A}{O2}{1.25}{1}{red}
\node at (2.1,-2.0) {$S_6(2)=6$};

\coordinate (O3) at (5.5,0);
\def\Rthree{1.45}

\HexShell{B}{O3}{\Rthree}{2}{red}

\pgfmathsetmacro{\Rthreei}{\Rthree/2}
\coordinate (O3i) at ($(O3)+(\Rthree-\Rthreei,0)$);
\HexShell{C}{O3i}{\Rthreei}{1}{blue}

\node at (5.5,-2.0) {$S_6(3)=15$};

\coordinate (O4) at (9.5,0);
\def\Rfour{1.65}

\HexShell{D}{O4}{\Rfour}{3}{red}

\pgfmathsetmacro{\Rfouri}{\Rfour*2/3}
\coordinate (O4i) at ($(O4)+(\Rfour-\Rfouri,0)$);
\HexShell{E}{O4i}{\Rfouri}{2}{blue}

\pgfmathsetmacro{\Rfourii}{\Rfour/3}
\coordinate (O4ii) at ($(O4)+(\Rfour-\Rfourii,0)$);
\HexShell{F}{O4ii}{\Rfourii}{1}{yellow}

\node at (9.5,-2.0) {$S_6(4)=28$};

\coordinate (O5) at (14,0);
\def\Rfive{1.82}

\HexShell{G}{O5}{\Rfive}{4}{red}

\pgfmathsetmacro{\Rfivei}{\Rfive*3/4}
\coordinate (O5i) at ($(O5)+(\Rfive-\Rfivei,0)$);
\HexShell{H}{O5i}{\Rfivei}{3}{blue}

\pgfmathsetmacro{\Rfiveii}{\Rfive/2}
\coordinate (O5ii) at ($(O5)+(\Rfive-\Rfiveii,0)$);
\HexShell{I}{O5ii}{\Rfiveii}{2}{yellow}

\pgfmathsetmacro{\Rfiveiii}{\Rfive/4}
\coordinate (O5iii) at ($(O5)+(\Rfive-\Rfiveiii,0)$);
\HexShell{J}{O5iii}{\Rfiveiii}{1}{green}

\node at (14,-2.0) {$S_6(5)=45$};

\end{tikzpicture}
\caption{The first five hexagonal numbers.}
\label{fig:hexagonal-numbers}
\end{figure}
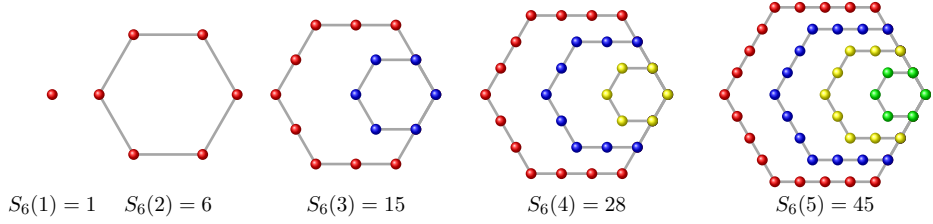

\noindent
This leads to the questions of determining the exact number of $\Fq$-rational points and finding polynomial families that produce such subsequences.

\medskip
\noindent
To determine these point counts, we first consider the Danielewski surface
$$
S_d=\{(\al,\be,\g)\in\A^3_{\Fq}\mid \al^d\be=P(\g)\}.
$$
It is enough to determine
$$
N_0=\#\{\g\in\Fq\mid P(\g)=0\},
$$
the number of roots of $P(Z)$ over $\Fq$.

\medskip
\noindent
For the double Danielewski surface
$$
S_{d_1,d_2}
=
\{(\al,\be,\g,\ka)\in\A^4_{\Fq}\mid
\al^{d_1}\be=P(\al,\g),~
\al^{d_2}\ka=Q(\al,\be,\g)\},
$$
the contribution from $\al\neq0$ can be computed directly, while the case $\al=0$ requires separate consideration. Hence, it is enough to determine
$$
N_0'
=
\#\{(\be,\g)\in\Fq^2\mid
P(0,\g)=0,~
Q(0,\be,\g)=0\}.
$$

\medskip
\noindent
Our aim is to count the $\Fq$-rational points directly from the equations defining the surfaces. We compute the numbers $N_0$ and $N_0'$ using algebraic and character-theoretic methods. Another approach is through the zeta-function \cite{Weil,swinnerton1967zeta}, which, however,  
involves counting points over all finite extensions $\mathbb{F}_{q^n}$ of a fixed base field. This requires substantially more counting information than is needed for our purpose and therefore we prefer a direct computation of $N_0$ and $N_0'$. 
\medskip

\noindent
To this end, we develop several methods for computing $N_0$ and $N_0'$. We use gcd-based root-counting methods, resultants, root-counting for special forms of $Q(0,Y,Z)$, additive and multiplicative character sums, Gauss sums, and the K\"onig--Rados theorem. These methods give explicit formulas in monomial, binomial, quadratic, and permutation-polynomial cases, as well as formulas in terms of ranks of circulant matrices. We also give a Macaulay2 algorithm \cite{M2} for computing the number of $\Fq$-rational points and verifying the formulas obtained in the paper.

\medskip
\noindent
The paper is organized as follows. In Section~\ref{S2}, we recall the finite-field notation and the preliminary results used later. Next, in Section~\ref{S3}, we show that it suffices to determine $N_0$, the number of roots of $P(Z)$ over $\Fq$, to obtain the point count of a Danielewski surface over $\Fq$. Further, in Section~\ref{S4}, we show that it suffices to determine $N_0'$, the number of $\Fq$-solutions of $P(0,Z)=0$ and $Q(0,Y,Z)=0$, to obtain the point count of a double Danielewski surface. Moreover, we obtain an additive-character formula in Theorem~\ref{thm:additive-master-double}, a formula in terms of ranks of circulant matrices in Theorem~\ref{thm:KR-general-N0}, an explicit formula for reduced equations of the form $Y^m=f(\g)$ in Proposition~\ref{prop:Ym-aZm}, and a formula for quadratic reduced equations for $q=p$ with $p$ odd in Theorem~\ref{thm:quadratic-N0}. In Section~\ref{S5}, we establish general bounds for $N_0'$. Thereafter, in Section~\ref{S6}, we give a Macaulay2 algorithm for computing and verifying the point counts. 
Finally, in Section~\ref{S7}, we give concluding remarks, examples whose point counts form subsequences of the square, hexagonal, octagonal, and decagonal number sequences as $q$ varies over prime powers, and directions for further work.

\section{Preliminaries}\label{S2}

\noindent
In this section, we recall basic results from finite field theory and character sums that will be used to count $\Fq$-rational points on Danielewski and double Danielewski surfaces.

\medskip
\noindent
Let $p$ be a prime number, let $r\geq 1$ be an integer, and let
$
\Fq:=\mathbb{F}_{p^r}
$
denote the finite field with $q=p^r$ elements. We write $\Fq^\times=\Fq\setminus\{0\}$ for the multiplicative group of nonzero elements of $\Fq$.

\medskip
\noindent
For affine spaces over $\Fq$, we use the notation
$
\A^3_{\Fq}(\Fq)=\Fq^3
~\text{and}~
\A^4_{\Fq}(\Fq)=\Fq^4.
$
The symbols $X,Y,Z,T$ denote coordinates, and $\al,\be,\g,\ka$ denote their values in $\Fq$. Thus, $(\al,\be,\g)\in\Fq^3$ and $(\al,\be,\g,\ka)\in\Fq^4$. The symbol $\sqcup$ denotes disjoint union.

\medskip
\noindent
Let $\overline{\Fq}$ denote an algebraic closure of $\Fq$. Since we count $\Fq$-rational points rather than points over $\overline{\Fq}$, the coordinates are required to lie in $\Fq$. For $a\in\overline{\Fq}$,
$$
a\in\Fq \quad\Longleftrightarrow\quad a^q-a=0.
$$
Equivalently,
$
U^q-U=\prod_{a\in\Fq}(U-a).
$
Thus, the relation $U^q-U=0$ restricts the value of $U$ to $\Fq$. This fact will be used in the gcd-based root-counting method and in the computational verification.

\medskip
\noindent
We next recall the character-theoretic results used later.

\begin{defi}
Let $G$ be a finite abelian group, written multiplicatively, with identity element $1_G$.
A \emph{character} of $G$ is a group homomorphism
$$
\chi : G \to \mathbb{C}^{\times},
$$
where $\mathbb{C}^{\times}$ denotes the multiplicative group of complex numbers of absolute value $1$. Thus
$$
\chi(g_1 g_2)=\chi(g_1)\chi(g_2)
\quad \text{for all } g_1,g_2\in G.
$$
In particular,
$$
\chi(g^{-1})=\chi(g)^{-1}=\overline{\chi(g)}
\quad \text{for all } g\in G,
$$
where the bar denotes complex conjugation.

\noindent
The \emph{trivial character} of $G$, denoted by $\chi_0$, is defined by
$$
\chi_0(g)=1 \quad \text{for all } g\in G.
$$
The set of all characters of $G$ is denoted by $\widehat G$.
\end{defi}

\begin{theorem}[\cite{LM13}, Theorem~5.4]\label{5.4}
Let $G$ be a finite abelian group. If $\chi$ is a nontrivial character of $G$, then
\begin{equation}\label{5.1}
\sum_{g\in G}\chi(g)=0.
\end{equation}
Further, if $g\in G$ with $g\neq 1_G$, then
\begin{equation}\label{5.2}
\sum_{\chi\in\widehat G}\chi(g)=0.
\end{equation}
\end{theorem}

\begin{theorem}[\cite{LM13}, Theorem~5.5]\label{5.5}
Let $G$ be a finite abelian group. Then the number of characters of $G$ equals $|G|$, that is,
$
|\widehat G|=|G|.
$
\end{theorem}

\noindent
Theorems~\ref{5.4} and~\ref{5.5} together yield the orthogonality relations for characters of $G$. For characters $\chi,\phi\in\widehat{G}$, one has
$$
\frac{1}{|G|}\sum_{g\in G}\chi(g)\overline{\phi(g)}
=
\begin{cases}
0, & \chi\neq\phi,\\
1, & \chi=\phi.
\end{cases}
$$
Indeed, the first case follows from \eqref{5.1} applied to the character $\chi\overline{\phi}$, while the second is immediate. Similarly, for $g,h\in G$,
$$
\frac{1}{|G|}\sum_{\chi\in\widehat G}\chi(g)\overline{\chi(h)}
=
\begin{cases}
0, & g\neq h,\\
1, & g=h.
\end{cases}
$$
\medskip
\noindent
Indeed, $\chi(g)\overline{\chi(h)}=\chi(gh^{-1})$, so the result follows from \eqref{5.2} and Theorem~\ref{5.5}.

\medskip
\noindent
These orthogonality relations will be used to express point counts in terms of character sums.

\medskip
\noindent
We now recall additive characters of finite fields. The absolute trace is used to define the canonical additive character used later in the point-counting formulas.

\begin{defi}
Let $F=\FF_{q^m}$ and let $K=\FF_q$. For $a\in F$, the \emph{trace} of $a$ from $F$ to $K$ is defined by
$$
\Tr_{F/K}(a)=a+a^q+a^{q^2}+\cdots+a^{q^{m-1}}.
$$
In the special case where $K=\FF_p$ is the prime subfield of $F$, then $\Tr_{F/K}(a)$ is called the \emph{absolute trace} of $a$ and is denoted by $\Tr_F(a)$.
\end{defi}

\begin{defi}
Let $\Tr:\FF_q\to\FF_p$ denote the absolute trace.
Define
\begin{equation}\label{eq:canonical-additive}
\chi_1(a)=e^{2\pi i\,\Tr(a)/p},
\qquad a\in\FF_q.
\end{equation}
The character $\chi_1$ is called the \emph{canonical additive character} of $\FF_q$. 
\end{defi}

\begin{theorem}[\cite{LM13}, Theorem~5.7]
Let $\chi_1$ denote the canonical additive character of $\FF_q$. For each $b\in\FF_q$, define
$$
\chi_b(a)=\chi_1(ba)\qquad(a\in\FF_q).
$$
Then $\chi_b$ is an additive character of $\FF_q$, and every additive character of $\FF_q$ is obtained in this way.
\end{theorem}

\begin{theorem}\label{5.6}
Let $\chi$ be a nontrivial additive character of $\FF_q$, and let $\chi_0$ denote the trivial additive character. For a polynomial $f(X)\in\FF_q[X]$ and an element $a\in\FF_q$, let $N$ denote the number of solutions to the equation $f(X)=a$ in $\FF_q$. Then
$$
N=
1+\frac{1}{q}\sum_{\chi\neq\chi_0}\overline{\chi(a)}\,S(\chi),
\qquad \text{where} \qquad
S(\chi):=\sum_{\al\in\FF_q}\chi\big(f(\al)\big).
$$
\end{theorem}

\begin{proof}
By the orthogonality relations for additive characters,
$$
\sum_{\chi\in\widehat{\FF_q}}\chi(b)=
\begin{cases}
q, & \text{if } b=0,\\
0, & \text{if } b\neq 0,
\end{cases}
$$
for all $b\in\FF_q$, where $\widehat{\FF_q}$ denotes the group of additive characters of $\FF_q$. Fix $\al\in\FF_q$. Applying the above orthogonality relation to $b=f(\al)-a$, we obtain
$$
\sum_{\chi\in\widehat{\FF_q}}\chi\bigl(f(\al)-a\bigr)=
\begin{cases}
q, & \text{if } f(\al)=a,\\
0, & \text{if } f(\al)\neq a.
\end{cases}
$$
Thus, varying $\al$ over $\Fq$ gives
\begin{align*}
N
&=\frac{1}{q}\sum_{\al\in\FF_q}\sum_{\chi\in\widehat{\FF_q}}
\chi\big(f(\al)-a\big) \\
&=\frac{1}{q}\sum_{\al\in\FF_q}
\Bigl[1+\sum_{\chi\neq\chi_0}\chi\big(f(\al)-a\big)\Bigr] \\
&=1+\frac{1}{q}\sum_{\chi\neq\chi_0}
\sum_{\al\in\FF_q}\chi\big(f(\al)\big)\overline{\chi(a)} \\
&=1+\frac{1}{q}\sum_{\chi\neq\chi_0}\overline{\chi(a)}\,S(\chi),
\end{align*}
which completes the proof.
\end{proof}

\noindent
We next recall the results on multiplicative characters and left circulant matrices used later.

\begin{defi}
A \emph{multiplicative character} on $\mathbb{F}_q$ is a map
$$
\psi:\mathbb{F}_q^\times \to \mathbb{C}^\times
$$
that satisfies
$$
\psi(ab)=\psi(a)\psi(b)
\qquad \text{for all } a,b\in \mathbb{F}_q^\times.
$$
\end{defi}

\begin{remark}
We extend a multiplicative character $\psi$ to $\FF_q$ by the convention
$$
\psi(0)=
\begin{cases}
1, & \text{if } \psi=\psi_0,\\
0, & \text{if } \psi\neq\psi_0,
\end{cases}
$$
where $\psi_0$ denotes the trivial multiplicative character.
\end{remark}

\begin{defi}
Let $q$ be odd. The \emph{quadratic character} $\eta$ of $\FF_q$ is defined by
$$
\eta(a)=
\begin{cases}
1, & \text{if } a\in(\FF_q^\times)^2,\\
-1, & \text{if } a\in\FF_q^\times\setminus(\FF_q^\times)^2,\\
0, & \text{if } a=0.
\end{cases}
$$
\end{defi}

\begin{defi}
Let $\psi$ be a multiplicative character  of $\FF_q$ and let $\chi$ be an additive character of $\FF_q$.
The \emph{Gauss sum} associated with $\psi$ and $\chi$ is defined by
$$
G(\psi,\chi)
=
\sum_{a\in\FF_q}\psi(a)\chi(a).
$$
\end{defi}

\begin{defi}[\cite{LM13}, Equation $6.1$]\label{leftcirculant} 
Let $f(X)=a_0+a_1X+a_2X^2+\cdots+a_{q-2}X^{q-2}\in\Fq[X].$ The \emph{left circulant matrix} associated with $f$ is the $(q-1)\times(q-1)$ matrix
$$
M=(a_{i+j})_{0\le i,j\le q-2},
$$
where the indices are taken modulo $q-1$.
Equivalently,
$$
M=
\begin{bmatrix}
a_0 & a_1 & a_2 & \cdots & a_{q-3} & a_{q-2} \\
a_1 & a_2 & a_3 & \cdots & a_{q-2} & a_0 \\
a_2 & a_3 & a_4 & \cdots & a_0 & a_1 \\
\vdots & \vdots & \vdots & \ddots & \vdots & \vdots \\
a_{q-2} & a_0 & a_1 & \cdots & a_{q-4} & a_{q-3}
\end{bmatrix}.
$$
\end{defi}
\medskip


\section{Number of \texorpdfstring{$\Fq$}{Fq}-Points on the Danielewski Surface}\label{S3}

\noindent
In this section, we determine the number of $\Fq$-rational points on the Danielewski surface
$$
S_d:\ X^dY=P(Z),
$$
where $d\geq1$ and $P(Z)\in\Fq[Z]$ with $\deg P\geq2$. Thus,
$$
S_d(\Fq)
=
\{(\al,\be,\g)\in\Fq^3\mid \al^d\be=P(\g)\}.
$$

\medskip
\noindent
We first decompose the set of $\Fq$-rational points of $S_d$ according to whether the $X$-coordinate vanishes or not. We fix the notation
$$
S_1 := \{(\al,\be,\g)\in S_d(\Fq)\mid \al=0\},
\qquad
S_2 := \{(\al,\be,\g)\in S_d(\Fq)\mid \al\neq0\}.
$$
Then
$$
S_d(\Fq)=S_1\sqcup S_2.
$$

\medskip
\noindent
If $\al=0$, the equation $\al^{d}\be=P(\g)$ reduces to $P(\g)=0$, and hence
$$
S_1=\{(0,\be,\g)\in\Fq^3 \mid P(\g)=0\}.
$$
If $\al\neq0$, then $\al\in\Fq^\times$, and the equation uniquely determines $\be=\al^{-d}P(\g)$, so
$$
S_2=\{(\al,\be,\g)\in\Fq^3 \mid \al\in\Fq^\times,\ \be=\al^{-d}P(\g)\}.
$$

\medskip
\noindent
We first count the points in $S_2$.

\begin{prop}\label{p1}
The number of $\Fq$-rational points of the Danielewski surface $S_d$ with $X\neq0$ is $|S_2| = q^2-q.$
\end{prop}

\begin{proof}
For $\al\in\Fq^\times$ and $\g\in\Fq$, the equation $\al^{d}\be=P(\g)$ uniquely determines $\be=\al^{-d}P(\g)\in\Fq$. Thus the map
$$
\Fq^\times\times\Fq \to S_2,
\qquad
(\al,\g)\to(\al,\al^{-d}P(\g),\g),
$$
is bijective. Since $|\Fq^\times|=q-1$ and $|\Fq|=q$, we obtain $|S_2|=q^2-q$.
\end{proof}

\medskip
\noindent
We now count the points in $S_1$. For this, we define the quantity
$$
N_0:=\lvert\{\g\in\Fq\mid P(\g)=0\}\rvert,
$$
which denotes the number of roots of $P(Z)$ in $\Fq$.
\begin{prop}\label{p2}
The number of $\Fq$-rational points of $S_d$ with $X=0$, that is, the number of points in $S_1$, is $|S_1| = q N_0.$
\end{prop}

\begin{proof}
If $\al=0$, the equation $\al^{d}\be=P(\g)$ forces $P(\g)=0$. For each $\g\in\Fq$ such that $P(\g)=0$, the coordinate $\be$ can be chosen arbitrarily in $\Fq$. Hence, for each root $\g$ of $P$ in $\Fq$, there are exactly $q$ corresponding points in $S_1$. Therefore, $|S_1|=qN_0$.
\end{proof}

\noindent
We now combine the preceding two counts.
\begin{remark}\label{dan}
Combining Propositions~\ref{p1} and~\ref{p2}, we obtain
$$
|S_d(\Fq)| = |S_1| + |S_2|
= qN_0 + q(q-1)
= q(q-1+N_0).
$$
\end{remark}

\begin{corollary}
If $P(Z)$ has no roots in $\Fq$, that is, $N_0=0$, then $|S_d(\Fq)| = q^2-q.$
\end{corollary}

\medskip
\noindent
Further, Remark~\ref{dan} shows that it suffices to compute $N_0$ to determine the number of $\Fq$-rational points on $S_d$. We first give a gcd expression for $N_0$.

\begin{prop}\label{gcdargument}
The quantity $N_0$ satisfies $$N_0=\deg\bigl(\gcd(P(Z),\,Z^q-Z)\bigr).$$
\end{prop}
\begin{proof}
Since
$$
Z^q-Z=\prod_{\g\in\Fq}(Z-\g),
$$
each element of $\Fq$ occurs as a root with multiplicity one. Hence,
$\gcd(P(Z),Z^q-Z)$ is the product of the factors $Z-\g$ for which $P(\g)=0$. Its degree is therefore $N_0$.
\end{proof}

\medskip
\noindent
An alternative expression for $N_0$ can be obtained using additive characters.
\begin{prop}\label{charactersum}
Let $\chi_0$ denote the trivial additive character of $\Fq$. Then 
$$
N_0 =
1 + \frac{1}{q}
\sum_{\chi \neq \chi_0} S_P(\chi),
\qquad  \text{where}
\qquad
S_P(\chi) := \sum_{\g \in \Fq} \chi(P(\g)).
$$
\end{prop}
\begin{proof} 
Applying Theorem~\ref{5.6} with $a=0$ and $f=P$, we obtain
$$
N_0
=
1+\frac{1}{q}\sum_{\chi\ne\chi_0}\overline{\chi(0)}\,S_P(\chi),
\qquad
S_P(\chi):=\sum_{\g\in\Fq}\chi(P(\g)).
$$
Since $\overline{\chi(0)}=1$ for every additive character $\chi$, we obtain
$$
N_0
=
1+\frac{1}{q}\sum_{\chi\ne\chi_0} S_P(\chi).
$$
\end{proof}

\noindent
We next obtain a formula for $N_0$ using the K\"onig--Rados theorem. We first recall the theorem.

\begin{theorem}[\cite{LM13}, K\"onig--Rados, Theorem $6.1$]\label{kongrados}
Let $f(X) = a_0 + a_1X + a_2X^2 + \cdots + a_{q-2}X^{q-2} \in \Fq[X].$ Then the number of nonzero solutions of the equation $f(X)=0$ in $\Fq$ is $q-1-\rho$, where $\rho$ is the rank of the left circulant matrix corresponding to $f(X)$.
\end{theorem}

\begin{theorem}\label{thm:KR-general-N0_1}
Let $q>2$, and let $P(Z)\in\Fq[Z]$. Let $R(Z)\in\Fq[Z]$ be the unique polynomial satisfying
$$
R(Z)\equiv P(Z)\pmod{Z^q-Z},
\qquad
\deg R\le q-1.
$$
Write
$$
R(Z)=a_0+a_1Z+\cdots+a_{q-2}Z^{q-2}+a_{q-1}Z^{q-1}.
$$
Define
$$
S(Z)
=
(a_0+a_{q-1})+a_1Z+\cdots+a_{q-2}Z^{q-2}.
$$
\noindent
Let $\rho$ denote the rank of the left circulant matrix associated with $S(Z)$. Then
$$
N_0
=
\#\{\g\in\Fq\mid P(\g)=0\}
=
q-1-\rho+\delta_{a_0},
$$
where
$$
\delta_{a_0}
=
\begin{cases}
1, & \text{if } a_0=0,\\
0, & \text{if } a_0\neq 0.
\end{cases}
$$
\end{theorem}
\begin{proof}
Since
$$
R(Z)\equiv P(Z)\pmod{Z^q-Z},
$$
we have $P(\g)=R(\g)$ for every $\g\in\Fq$. Hence
$$
N_0=\#\{\g\in\Fq\mid R(\g)=0\}.
$$
\noindent
Note that, the point $\g=0$ contributes $\delta_{a_0}$, because $R(0)=a_0$.
\medskip
\noindent
Further, for $\g\in\Fq^\times$, we have $\g^{q-1}=1$. Therefore
\begin{align*}
R(\g)
&= a_0+a_1\g+\cdots+a_{q-2}\g^{q-2}+a_{q-1}\g^{q-1} \\
&= (a_0+a_{q-1})+a_1\g+\cdots+a_{q-2}\g^{q-2} \\
&= S(\g).  
\end{align*}
\noindent
Thus the number of nonzero roots of $R(Z)$ in $\Fq$ is the same as the number of nonzero roots of $S(Z)$ in $\Fq$. Since $\deg S\le q-2$, the K\"onig--Rados theorem applies to $S(Z)$ and gives that the number of nonzero roots of $S(Z)$ in $\Fq$ is
$$
q-1-\rho,
$$
where $\rho$ is the rank of the left circulant matrix associated with $S(Z)$. Hence
$$
N_0
=
(q-1-\rho)+\delta_{a_0}.
$$
\end{proof}

\noindent
We now study a class of cases in which $P(Z)$ is a homogeneous polynomial.

\begin{theorem}
Let $P(Z)\in\Fq[Z]$ be a nonzero homogeneous polynomial of degree $s\ge 2$. Then $N_0=1.$
\end{theorem}

\begin{proof}
Since $P(Z)$ is a nonzero homogeneous polynomial in one variable, there exists $a\in\Fq^\times$ such that $P(Z)=aZ^s.$ Thus $P(Z)=0$ has exactly one solution in $\Fq$, namely $Z=0$. Therefore $N_0=1.$
\end{proof}

\noindent
We next consider the case where $P(Z)$ is monomial or binomial in $Z$. We use the following standard result.

\begin{lemma}[\cite{IrelandRosen1990}, Proposition 7.1.2]\label{lem:binomial-solvability} 
Let $a\in \Fq^\times$, and let $d_m=\gcd(m,q-1).$ Then the equation $X^m=a$ has solutions in $\Fq$ if and only if $a^{(q-1)/d_m}=1.$ If there are solutions, then there are exactly $d_m$ solutions. 
\end{lemma}

\begin{theorem}\label{thm:single-binomial-N0}
Suppose that $P(Z)=Z^m-a,$ where $m\ge 1$ and $a\in\Fq$. Let $d_m=\gcd(m,q-1).$ Then
$$
N_0=
\begin{cases}
1, & a=0,\\
0, & a\neq 0 \text{ and } a^{(q-1)/d_m}\neq 1,\\
d_m, & a\neq 0 \text{ and } a^{(q-1)/d_m}=1.
\end{cases}
$$
\end{theorem}

\begin{proof}
If $a=0$, then the only solution is $\g=0$, so $N_0=1$. Now assume that $a\neq 0$. By Lemma~\ref{lem:binomial-solvability}, the equation $Z^m=a$ has solutions in $\Fq$ if and only if $a^{(q-1)/d_m}=1.$ If this condition holds, then there are exactly $d_m$ solutions; otherwise there are no solutions. This gives us the required value of $N_0$ in all the cases.
\end{proof}

\begin{corollary}
If $P(Z)=Z^m$ for some $m\geq2$ is a monomial, then $N_0=1.$
\end{corollary}

\begin{proof}
The proof follows from Theorem \ref{thm:single-binomial-N0}.
\end{proof}

\begin{theorem}\label{thm:single-binomial-two-term}
Suppose that $P(Z)=aZ^m-bZ^n,$ ~ $0\le n<m\le q-1,$ where $a,b\in\Fq$. Then the following hold.
\begin{enumerate}
\item If $n>0$, set $d_{m-n}=\gcd(m-n,q-1)$. Then
$$
N_0=
\begin{cases}
q, & a=0,\ b=0,\\
1+d_{m-n}, & ab\neq0 \text{ and }
\left(\dfrac{b}{a}\right)^{(q-1)/d_{m-n}}=1,\\
1, & \text{otherwise}.
\end{cases}
$$

\item If $n=0$, set $d_m=\gcd(m,q-1)$. Then
$$
N_0=
\begin{cases}
q, & a=0,\ b=0,\\
1, & a\neq 0,\ b=0,\\
d_m, & ab\neq 0 \text{ and } \left(\dfrac{b}{a}\right)^{(q-1)/d_m}=1,\\
0, & \text{otherwise}.
\end{cases}
$$
\end{enumerate}
\end{theorem}

\begin{proof}
\begin{enumerate}
\item Suppose that $n>0$. Then
$
aZ^m-bZ^n=Z^n(aZ^{m-n}-b).
$

\medskip
\noindent
Thus, $\g=0$ is always a solution. If $a=b=0$, then all elements of $\Fq$ are solutions. If exactly one of $a,b$ is zero, then no nonzero solution occurs, so only $\g=0$ remains. If $ab\neq0$, then the nonzero solutions satisfy
$$
Z^{m-n}=\frac{b}{a}.
$$
By Lemma~\ref{lem:binomial-solvability}, this equation has exactly
$d_{m-n}=\gcd(m-n,q-1)$ solutions if and only if
$$
\left(\frac{b}{a}\right)^{(q-1)/d_{m-n}}=1,
$$
and has no solution otherwise. This gives the formula in $(i)$.

\item Suppose that $n=0$. Then $P(Z)=aZ^m-b$.

\medskip
\noindent
If $a=b=0$, then all elements of $\Fq$ are solutions. If $a=0$ and $b\neq0$, then there is no solution. If $a\neq0$ and $b=0$, then only $\g=0$ is a solution. If $ab\neq0$, then
$$
Z^m=\frac{b}{a}.
$$
By Lemma~\ref{lem:binomial-solvability}, this equation has exactly
$d_m=\gcd(m,q-1)$ solutions if and only if
$$
\left(\frac{b}{a}\right)^{(q-1)/d_m}=1,
$$
and has no solution otherwise. This gives the formula in $(ii)$.
\end{enumerate}
\end{proof}
\noindent
\noindent
Taking $a=1$ in Theorem~\ref{thm:single-binomial-two-term} gives the following.

\begin{corollary}\label{cor:single-binomial-a1}
Suppose that $P(Z)=Z^m-bZ^n$, $0\leq n<m\leq q-1$, where $b\in\Fq$. Then the following hold.

\begin{enumerate}
\item
If $n>0$ and $d_{m-n}=\gcd(m-n,q-1)$, then
$$
N_0=
\begin{cases}
1+d_{m-n}, & b\neq0 \text{ and } b^{(q-1)/d_{m-n}}=1,\\
1, & \text{otherwise}.
\end{cases}
$$

\item
If $n=0$ and $d_m=\gcd(m,q-1)$, then
$$
N_0=
\begin{cases}
1, & b=0,\\
d_m, & b\neq0 \text{ and } b^{(q-1)/d_m}=1,\\
0, & \text{otherwise}.
\end{cases}
$$
\end{enumerate}
\end{corollary}
\noindent
Further, as a direct consequence of the K\"onig--Rados theorem, we obtain the following expression for $N_0$ when $P(Z)$ has degree at most $q-2$.

\begin{prop}\label{prop:KR-count}
Let $q>3$ and let $P(Z)=a_0+a_1Z+a_2Z^2+\cdots+a_{q-2}Z^{q-2}\in\Fq[Z].$ Then $N_0 = q-1-\rho_{P}+\delta_{a_0},$ where $\rho_{P}$ denotes the rank of the left circulant matrix associated with $P(Z)$, and
$$
\delta_{a_0}=
\begin{cases}
1, & \text{if } a_0=0,\\
0, & \text{if } a_0\neq0.
\end{cases}
$$
\end{prop}

\begin{proof}
Since $\deg P\leq q-2$, we have $R(Z)=S(Z)=P(Z)$ in Theorem~\ref{thm:KR-general-N0_1}. Hence the stated formula follows.
\end{proof}

\noindent
Furthermore, for prime fields, multiplicative characters and the quadratic character give further formulas for $N_0$. We first recall the following two results from \cite{IrelandRosen1990}.

\begin{lemma}[\cite{IrelandRosen1990}, Exercise 8.1]\label{multcharsol}
Let $p$ be a prime, let $m\geq1$, and let $a\in\F_p^\times$. Set
$d_m=\gcd(m,p-1)$. Then the number of solutions in $\FF_p$ to the equation $X^m=a$ is given by $\sum_{\psi^{d_m}=\psi_0}\psi(a),$ where the sum is taken over all multiplicative characters $\psi$ of $\FF_p$ satisfying $\psi^{d_m}=\psi_0$, and where $\psi_0$ denotes the trivial multiplicative character.
\end{lemma}

\begin{lemma}[\cite{IrelandRosen1990}, Exercise 5.3]
\label{quadraticformula}
Let $p$ be an odd prime and let $a,b,c\in\F_p$ with $a\neq0$. Then the number of solutions in $\F_p$ to
$
aX^2+bX+c=0
$
is
$
1+\eta(\Delta),
$
where $\Delta=b^2-4ac$ and $\eta$ denotes the quadratic character of $\F_p$.
\end{lemma}

\medskip
\noindent
Combining the above lemmas, we obtain explicit expressions for $N_0$ in the following situations.

\begin{theorem}
Assume that $q=p$ is an odd prime. Then the following hold.

\begin{enumerate}[label=(\roman*)]
\item
Consider the binomial  $P(Z)=Z^n-bZ^m,$ where $0\le m<n$ and $b\in\F_p$. Let $d_{n-m}=\gcd(n-m,p-1)$ and define
$$
\delta_m=
\begin{cases}
0, & m=0,\\
1, & m>0.
\end{cases}
$$
Then
$$
N_0=
\begin{cases}
1, & b=0,\\[1mm]
\displaystyle
\delta_m+\sum_{\psi^{d_{n-m}}=\psi_0}\psi(b),
& b\neq0,
\end{cases}
$$
where the sum is taken over all multiplicative characters $\psi$ of $\F_p^\times$ satisfying $\psi^{d_{n-m}}=\psi_0$.

\item For the trinomial polynomial $P(Z)=aZ^2+bZ+c,$ where $a,b,c\in\F_p$ with $a\neq 0$, we have $N_0=1+\eta(\Delta),$ where $\Delta=b^2-4ac$ and $\eta$ denotes the quadratic character of $\F_p$.
\end{enumerate}
\end{theorem}

\begin{proof}
The proof of part $(i)$ follows from Lemma~\ref{multcharsol} and part $(ii)$  from Lemma~\ref{quadraticformula} respectively.
\end{proof}

\section{Number of \texorpdfstring{$\Fq$}{Fq}-Points on the Double Danielewski Surface}{\label{S4}}

\noindent
In this section, we determine the number of $\Fq$-rational points on double Danielewski surfaces. As in the Danielewski case, we separate the cases $X=0$ and $X\neq0$. This reduces the problem to the computation of $N_0'$, which is studied in the following subsections.

\subsection{General decomposition and the quantity \texorpdfstring{$N_0'$}{N0'}}

\noindent
Consider the double Danielewski surface
$$
S_{d_1,d_2}:
\quad
X^{d_1}Y=P(X,Z),
\qquad
X^{d_2}T=Q(X,Y,Z),
$$
where $d_1,d_2\geq1$, $P(X,Z)\in\Fq[X,Z]$ is monic in $Z$, and $Q(X,Y,Z)\in\Fq[X,Y,Z]$ is monic in $Y$, with
$$
\deg_Z P(X,Z)\geq2,
\qquad
\deg_Y Q(X,Y,Z)\geq2.
$$
Thus,
$$
S_{d_1,d_2}(\Fq)
=
\Bigl\{
(\al,\be,\g,\ka)\in\Fq^4
\mid
\al^{d_1}\be=P(\al,\g),\
\al^{d_2}\ka=Q(\al,\be,\g)
\Bigr\}.
$$

\medskip
\noindent
As in the Danielewski case, we decompose $S_{d_1,d_2}(\Fq)$ according to whether the $X$-coordinate vanishes or not. Define
$$
S_1'
=
\{(\al,\be,\g,\ka)\in S_{d_1,d_2}(\Fq)\mid \al=0\},
\qquad
S_2'
=
\{(\al,\be,\g,\ka)\in S_{d_1,d_2}(\Fq)\mid \al\neq0\}.
$$
Then
$$
S_{d_1,d_2}(\Fq)=S_1'\sqcup S_2'.
$$

\medskip
\noindent
If $\al=0,$ then the defining equations reduce to $P(0,\g)=0,$ and $Q(0,\be,\g)=0.$ Hence
$$
S_1'
=
\{(0,\be,\g,\ka)\in \Fq^4\mid P(0,\g)=0,\;Q(0,\be,\g)=0\}.
$$
In particular, once $(\be,\g)$ satisfies these two equations, the coordinate $\ka$ may be chosen arbitrarily in $\Fq.$

\noindent
We first count the points in $S_2'$.

\begin{prop}\label{prop:double-nonzero-x}
The number of $\Fq$-rational points of $S_{d_1,d_2}$ with $X\neq 0$ is $|S_2'|=q^2-q.$
\end{prop}

\begin{proof}
For each $(\al,\g)\in \Fq^\times\times\Fq$, the defining equations uniquely determine
$$
\be=\al^{-d_1}P(\al,\g),
\qquad
\ka=\al^{-d_2}Q\bigl(\al,\al^{-d_1}P(\al,\g),\g\bigr).
$$
Hence $(\al,\g)\mapsto (\al,\be,\g,\ka)$ defines a bijection
$$
\Fq^\times\times\Fq \xrightarrow{\sim} S_2'.
$$
Therefore
$$
|S_2'|=|\Fq^\times|\,|\Fq|=q^2-q.
$$
\end{proof}

\medskip
\noindent
We next count the points in $S_1'$. Define
$$
N_0'
:=
\#\{(\be,\g)\in\Fq^2
\mid
P(0,\g)=0,\ Q(0,\be,\g)=0\},
$$
the number of pairs $(\be,\g)\in\Fq^2$ satisfying the reduced system at $X=0$.

\begin{prop}\label{prop:double-zero-x}
The number of $\Fq$-rational points of $S_{d_1,d_2}$ with $X=0$ is $|S_1'|=qN_0'.$
\end{prop}

\begin{proof}
If $\al=0$, then the reduced system is $P(0,\g)=0$ and 
$Q(0,\be,\g)=0.$ Thus there are exactly $N_0'$ choices for the pair $(\be,\g),$ and for each such pair the coordinate $\ka$ may be chosen arbitrarily in $\Fq.$ Hence $|S_1'|=qN_0'.$
\end{proof}

\begin{corollary}\label{cor:double-total-count}
The total number of $\Fq$-rational points on $S_{d_1,d_2}$ is
$$|S_{d_1,d_2}(\Fq)|=q\bigl(q-1+N_0'\bigr).$$
\end{corollary}

\begin{proof}
By Propositions~\ref{prop:double-nonzero-x} and~\ref{prop:double-zero-x},
$$
|S_{d_1,d_2}(\Fq)|
=
|S_1'|+|S_2'|
=
qN_0'+q(q-1)
=
q\bigl(q-1+N_0'\bigr).
$$
\end{proof}

\begin{remark}
The formula
$$
|S_{d_1,d_2}(\Fq)|=q(q-1+N_0')
$$
shows that whenever $N_0'$ is of the form $aq+b$, for constants $a$ and $b$ in $\Fq$, the point count is a quadratic polynomial in $q$. Since polygonal numbers are also given by quadratic polynomials in their index, this gives a natural connection between these point counts and polygonal-number sequences. We return to this connection in Section~\ref{S7}.
\end{remark}
\medskip
\noindent
Corollary~\ref{cor:double-total-count} shows that it suffices to compute $N_0'$ to determine the number of $\Fq$-rational points on $S_{d_1,d_2}$. We first express $N_0'$ in terms of the roots of $P(0,Z)$.

\begin{prop}\label{thm:rootwise-reduction}
Let $\g_1,\dots,\g_k\in\Fq$ be the distinct roots of $P(0,Z)$. For each $i=1,\dots,k$, let
$$
j_i
=
\#\{\be\in\Fq\mid Q(0,\be,\g_i)=0\}.
$$
Then
$$
N_0'=j_1+\cdots+j_k.
$$
\end{prop}

\begin{proof}
For each root $\g_i$ of $P(0,Z)$, there are exactly $j_i$ elements $\be\in\Fq$ such that
$Q(0,\be,\g_i)=0$. Hence,
$$
N_0'=\sum_{i=1}^k j_i.
$$
\end{proof}

\medskip
\noindent
Thus, the computation of $N_0'$ reduces to determining, for each root $\g$ of $P(0,Z)$, the number of roots of $Q(0,Y,\g)$ in $\Fq$.

\subsection{General methods for computing \texorpdfstring{$\bm{N_0'}$}{N0'}}

\noindent
We now give general methods for computing $N_0'$. These will be applied in the subsequent subsections to different forms of the reduced polynomial $Q(0,Y,\g)$.

\medskip
\noindent
We first use a gcd-based root count. Since
$$
Y^q-Y=\prod_{\be\in\Fq}(Y-\be),
$$
for every $f(Y)\in\Fq[Y]$, the number of distinct roots of $f(Y)$ in $\Fq$ is
$$
\deg\bigl(\gcd(f(Y),Y^q-Y)\bigr).
$$

\begin{prop}\label{prop:gcd-method-double}
Let $ P(X,Z)=XP_1(X,Z)+P_2(Z)$ and $Q(X,Y,Z)=XQ_1(X,Y,Z)+Q_2(Y,Z),$ where $Q_2(Y,Z)\in\Fq[Y,Z]$ is monic in $Y$. 
Then
$$
N_0' =
\sum_{\substack{\g\in\Fq\\ P_2(\g)=0}}
\deg(\gcd(Q_2(Y,\g),Y^q-Y)).
$$
\end{prop}

\begin{proof}
For each $\g\in\Fq$ satisfying $P(0,\g)=0$, the number of $\be\in\Fq$ such that $Q(0,\be,\g)=0$ is exactly
$$
\deg(\gcd(Q(0,Y,\g),Y^q-Y)).
$$
The result now follows from Theorem~\ref{thm:rootwise-reduction}.
\end{proof}

\noindent
As an immediate consequence of Proposition~\ref{prop:gcd-method-double} and Corollary~\ref{cor:double-total-count}, we obtain:

\begin{corollary}\label{cor:no-root-P-zero}
If $P(0,Z)$ has no root in $\Fq$, then
$
N_0'=0
~\text{and}~
|S_{d_1,d_2}(\Fq)|=q^2-q.
$
\end{corollary}

\medskip
\noindent
A different way to compute $N_0'$ is to use resultants to eliminate the variable $Y$ from the equation $Q(0,Y,Z)=0$.

\begin{defi}
Let $K$ be a field and $f(X,Y)$ and $g(X,Y)$ be polynomials in $K[X,Y]$, regarded as elements of $(K[Y])[X].$ Write
$$
f(X,Y)=a_0(Y)X^l+a_1(Y)X^{l-1}+\cdots+a_l(Y),
~~
a_0(Y)\neq 0,~~l>0
$$
and
$$
g(X,Y)=b_0(Y)X^m+b_1(Y)X^{m-1}+\cdots+b_m(Y),
~~
b_0(Y)\neq 0,~~m>0.
$$
The resultant of $f$ and $g$ with respect to $X$, denoted by $\operatorname{Res}_X(f,g)$, is the determinant of the $(l+m)\times(l+m)$ Sylvester matrix associated with $f$ and $g$.
\end{defi}

\begin{lemma}[\cite{CoxLittleOShea2005}, Equation (1.4)]
\label{lem:resultant-product}
Let
$$
f(X)=a_0X^l+a_1X^{l-1}+\cdots+a_l,
\qquad
g(X)=b_0X^m+b_1X^{m-1}+\cdots+b_m,
$$
with $a_0b_0\neq0$ and $l,m>0$. If the roots of $f$ are
$\xi_1,\dots,\xi_l$ and the roots of $g$ are $\eta_1,\dots,\eta_m$,
in an extension field, then
$$
\operatorname{Res}(f,g)
=
a_0^{\,m}b_0^{\,l}
\prod_{i=1}^l\prod_{j=1}^m(\xi_i-\eta_j)
=
a_0^{\,m}\prod_{i=1}^l g(\xi_i)
=
(-1)^{lm}b_0^{\,l}\prod_{j=1}^m f(\eta_j).
$$
\end{lemma}

\begin{theorem}\label{thm:resultant-method-double}
Let
$
R_Y(Z)
=
\operatorname{Res}_Y\bigl(Y^q-Y,Q(0,Y,Z)\bigr)
\in\Fq[Z],
$
and define
$$
A
=
\{\g\in\Fq\mid P(0,\g)=0,\ R_Y(\g)=0\}.
$$
For each $\g\in A$, let
$$
m_\g
=
\#\{\be\in\Fq\mid Q(0,\be,\g)=0\}.
$$
Then
$$
N_0'=\sum_{\g\in A}m_\g.
$$
\end{theorem}

\begin{proof}
Since
$$
Y^q-Y=\prod_{\be\in\Fq}(Y-\be),
$$
applying Lemma~\ref{lem:resultant-product} gives
$$
R_Y(Z)
=
\prod_{\be\in\Fq}Q(0,\be,Z).
$$
Hence, if $Q(0,\be,\g)=0$, then $R_Y(\g)=0$. Therefore, every pair
$(\be,\g)$ counted by $N_0'$ has $\g\in A$. For each $\g\in A$, there are
exactly $m_\g$ elements $\be\in\Fq$ such that $Q(0,\be,\g)=0$. Hence
$$
N_0'
=
\sum_{\g\in A}m_\g.
$$
\end{proof}

\begin{corollary}\label{cor:resultant-QY}
Let $Q(0,Y,Z)=Q_1(Y)\in\Fq[Y]$. Then
$\operatorname{Res}_Y(Y^q-Y,Q_1(Y))\in\Fq^\times$
if and only if $Q_1(Y)$ has no root in $\Fq$.
If these equivalent conditions hold, then $N_0'=0$.
\end{corollary}

\begin{proof}
By Lemma~\ref{lem:resultant-product},
$$
\operatorname{Res}_Y(Y^q-Y,Q_1(Y))
=
\prod_{\be\in\Fq}Q_1(\be).
$$
Hence, the resultant is nonzero if and only if $Q_1(\be)\neq0$ for every
$\be\in\Fq$, equivalently, if and only if $Q_1(Y)$ has no root in $\Fq$.
In this case, the system
$$
P(0,\g)=0,
\qquad
Q_1(\be)=0
$$
has no solution in $\Fq^2$, and hence $N_0'=0$.
\end{proof}

\medskip
\noindent
A further approach expresses $N_0'$ in terms of additive character sums.

\begin{theorem}\label{thm:additive-master-double}
If $P(0,Z)$ has $r$ distinct roots in $\Fq$, then
\begin{equation}\label{eq:N0-master-double}
N_0'=r+
\frac{1}{q}
\sum_{\substack{\g\in\Fq\\ P(0,\g)=0}}
\sum_{s\in\Fq^\times}
\sum_{\be\in\Fq}
\chi_1\!\bigl(s\,Q(0,\be,\g)\bigr),
\end{equation}
where $\chi_1$ denotes a fixed nontrivial additive character of $\Fq$.
\end{theorem}

\begin{proof}
Let $\g_1,\dots,\g_r$ be the distinct roots of $P(0,Z)$. By Theorem~\ref{thm:rootwise-reduction},
$$
N_0'
=
\sum_{\substack{\g\in\Fq\\ P(0,\g)=0}}
\#\{\be\in\Fq\mid Q(0,\be,\g)=0\}.
$$
Fix such a root $\g_i$. By the orthogonality relation for additive characters,
$$
\frac{1}{q}\sum_{s\in\Fq}\chi_1(sa)
=
\begin{cases}
1, & \text{if } a=0,\\
0, & \text{if } a\neq 0,
\end{cases}
\qquad a\in\Fq.
$$
Applying this with $a=Q(0,\be,\g_i)$, we obtain
$$
\#\{\be\in\Fq\mid Q(0,\be,\g_i)=0\}
=
\frac{1}{q}\sum_{\be\in\Fq}\sum_{s\in\Fq}\chi_1\!\bigl(s\,Q(0,\be,\g_i)\bigr).
$$
Summing over all roots $\g_i$ of $P(0,Z)$, we get
$$
N_0'
=
\frac{1}{q}
\sum_{s\in\Fq}
\sum_{\substack{\g\in\Fq\\ P(0,\g)=0}}
\sum_{\be\in\Fq}
\chi_1\!\bigl(s\,Q(0,\be,\g)\bigr).
$$
The term $s=0$ contributes
$$
\frac{1}{q}\,rq=r.
$$
Separating this term gives \eqref{eq:N0-master-double}.
\end{proof}

\begin{remark}
\noindent
One can observe using character sums that if $P(0,Z)$ has $r$ distinct roots in $\Fq$ and
$Q(0,\be,\g)\neq0$ for every $\be\in\Fq$ and every root $\g$ of $P(0,Z)$, then
$$
\sum_{s\in\Fq^\times}\chi_1\!\bigl(s\,Q(0,\be,\g)\bigr)=-1.
$$
Therefore
$$
\sum_{\substack{\g\in\Fq\\ P(0,\g)=0}}
\sum_{s\in\Fq^\times}
\sum_{\be\in\Fq}
\chi_1\!\bigl(s\,Q(0,\be,\g)\bigr)
=
-rq.
$$
Substituting this into \eqref{eq:N0-master-double}, we obtain
$$
N_0'=r+\frac{1}{q}(-rq)=0.
$$
\end{remark}

\begin{corollary}
Let $P(0,Z)$ have $r$ distinct roots in $\Fq$. If, in Theorem~\ref{thm:additive-master-double}, we take $\chi_1$ to be the canonical additive character of $\Fq$, then
$$
N_0'
=
r+\frac{1}{q}
\sum_{\substack{\g\in\Fq\\ P(0,\g)=0}}
\sum_{s\in\Fq^\times}
\sum_{\be\in\Fq}
\zeta_p^{\,\Tr(s\,Q(0,\be,\g))},
$$
where $\zeta_p=e^{2\pi i/p}$. In particular, if $q=2^n$, then
$$
N_0'
=
r+\frac{1}{2^n}
\sum_{\substack{\g\in\Fq\\ P(0,\g)=0}}
\sum_{s\in\Fq^\times}
\sum_{\be\in\Fq}
(-1)^{\Tr(s\,Q(0,\be,\g))}.
$$
\end{corollary}

\medskip
\noindent
We also obtain a character-sum formula that treats the two equations
$P(0,Z)=0$ and $Q(0,Y,Z)=0$ simultaneously.

\begin{theorem}\label{thm:N0-double-char}
For additive characters $\chi$ and $\phi$ of $\Fq$, define
$$
S_P(\chi)
=
\sum_{\g\in\Fq}\chi\bigl(P(0,\g)\bigr),
$$
$$
S_Q(\phi)
=
\sum_{\be,\g\in\Fq}\phi\bigl(Q(0,\be,\g)\bigr),
$$
and
$$
S_{P,Q}(\chi,\phi)
=
\sum_{\be,\g\in\Fq}
\chi\bigl(P(0,\g)\bigr)\phi\bigl(Q(0,\be,\g)\bigr).
$$
Then
$$
N_0'
=
1
+
\frac{1}{q}\sum_{\chi\neq\chi_0}S_P(\chi)
+
\frac{1}{q^2}\sum_{\phi\neq\phi_0}S_Q(\phi)
+\frac{1}{q^2}\sum_{\substack{\chi\neq\chi_0\\ \phi\neq\phi_0}}S_{P,Q}(\chi,\phi),
$$
where $\chi_0$ and $\phi_0$ denote the trivial additive characters.
\end{theorem}

\begin{proof}
Using the orthogonality relation for additive characters, we have
$$
\frac{1}{q}\sum_{\chi}\chi\bigl(P(0,\g)\bigr)
=
\begin{cases}
1, & \text{if } P(0,\g)=0,\\
0, & \text{if } P(0,\g)\neq 0,
\end{cases}
$$
\noindent
and
$$
\frac{1}{q}\sum_{\phi}\phi\bigl(Q(0,\be,\g)\bigr)
=
\begin{cases}
1, & \text{if } Q(0,\be,\g)=0,\\
0, & \text{if } Q(0,\be,\g)\neq 0,
\end{cases}
$$
where $\chi$ and $\phi$ run over all additive characters of $\Fq.$

\noindent
Therefore,
$$
N_0'
=
\sum_{\be,\g\in\Fq}
\left(
\frac{1}{q}\sum_{\chi}\chi\bigl(P(0,\g)\bigr)
\right)
\left(
\frac{1}{q}\sum_{\phi}\phi\bigl(Q(0,\be,\g)\bigr)
\right).
$$

\noindent
Hence
$$
N_0'
=
\frac{1}{q^2}
\sum_{\be,\g\in\Fq}
\sum_{\chi}\sum_{\phi}
\chi\bigl(P(0,\g)\bigr)\phi\bigl(Q(0,\be,\g)\bigr).
$$
\noindent
Interchanging the order of summation, we obtain
$$
N_0'
=
\frac{1}{q^2}
\sum_{\chi}\sum_{\phi}
\sum_{\be,\g\in\Fq}
\chi\bigl(P(0,\g)\bigr)\phi\bigl(Q(0,\be,\g)\bigr).
$$

\noindent
We now separate the sums according to whether $\chi$ and $\phi$ are trivial or nontrivial.

\medskip
\noindent
If $\chi=\chi_0$ and $\phi=\phi_0$, then
$$
\frac{1}{q^2}
\sum_{\be,\g\in\Fq}
\chi_0\bigl(P(0,\g)\bigr)\phi_0\bigl(Q(0,\be,\g)\bigr)
=
\frac{1}{q^2}\sum_{\be,\g\in\Fq}1
=
1.
$$

\noindent
If $\chi\neq\chi_0$ and $\phi=\phi_0$, then
$$
\frac{1}{q^2}
\sum_{\chi\neq\chi_0}
\sum_{\be,\g\in\Fq}
\chi\bigl(P(0,\g)\bigr)
=
\frac{1}{q^2}
\sum_{\chi\neq\chi_0}
\left(\sum_{\g\in\Fq}\chi\bigl(P(0,\g)\bigr)\right)
\left(\sum_{\be\in\Fq}1\right)
=
\frac{1}{q}\sum_{\chi\neq\chi_0}S_P(\chi).
$$

\noindent
If $\chi=\chi_0$ and $\phi\neq\phi_0$, then
$$
\frac{1}{q^2}
\sum_{\phi\neq\phi_0}
\sum_{\be,\g\in\Fq}
\phi\bigl(Q(0,\be,\g)\bigr)
=
\frac{1}{q^2}\sum_{\phi\neq\phi_0}S_Q(\phi).
$$

\noindent
Finally, if $\chi\neq\chi_0$ and $\phi\neq\phi_0,$ then
$$
\frac{1}{q^2}
\sum_{\substack{\chi\neq\chi_0\\ \phi\neq\phi_0}}
\sum_{\be,\g\in\Fq}
\chi\bigl(P(0,\g)\bigr)\phi\bigl(Q(0,\be,\g)\bigr)
=
\frac{1}{q^2}
\sum_{\substack{\chi\neq\chi_0\\ \phi\neq\phi_0}}
S_{P,Q}(\chi,\phi).
$$

\noindent
Combining these four contributions, we get
$$
N_0'
=
1
+
\frac{1}{q}\sum_{\chi\neq\chi_0}S_P(\chi)
+
\frac{1}{q^2}\sum_{\phi\neq\phi_0}S_Q(\phi)
+
\frac{1}{q^2}\sum_{\substack{\chi\neq\chi_0\\ \phi\neq\phi_0}}S_{P,Q}(\chi,\phi).
$$
\noindent
This completes the proof.
\end{proof}

\medskip
\noindent
We next obtain a general formula for $N_0'$ using the K\"onig--Rados theorem.

\begin{theorem}
\label{thm:KR-general-N0}
Let $P(0,Z)$ have $r$ distinct roots $\gamma_1,\dots,\gamma_r$ in $\Fq$. Let $R(Y,Z)\in \Fq[Y,Z]$ be the unique polynomial satisfying
$$
R(Y,Z)\equiv Q(0,Y,Z)\pmod{Y^q-Y},
\qquad
\deg_Y R\le q-1.
$$
Write
$$
R(Y,Z)=\sum_{i=0}^{q-1} a_i(Z)Y^i,
\qquad
R_{\gamma_j}(Y):=R(Y,\gamma_j)=\sum_{i=0}^{q-1} a_i(\gamma_j)Y^i.
$$
After reindexing if necessary, assume that
$$
a_0(\gamma_j)=0 \quad \text{for } 1\le j\le t,
\qquad
a_0(\gamma_j)\neq 0 \quad \text{for } t+1\le j\le r.
$$
For each $j=1,\dots,r$, define
$$
S_{\gamma_j}(Y)=
\begin{cases}
\dfrac{R_{\gamma_j}(Y)}{Y}, & \text{if } a_0(\gamma_j)=0,\\[1ex]
\bigl(a_0(\gamma_j)+a_{q-1}(\gamma_j)\bigr)
+a_1(\gamma_j)Y+\cdots+a_{q-2}(\gamma_j)Y^{q-2},
 & \text{if } a_0(\gamma_j)\neq 0.
\end{cases}
$$
Let $\rho_{\gamma_j}$ denote the rank of the left circulant matrix associated with $S_{\gamma_j}(Y)$. Then
$$
N_0'=r(q-1)-\sum_{j=1}^r \rho_{\gamma_j}+t.
$$
\end{theorem}

\begin{proof}
For each $j=1,\dots,r$, let
$$
N_{\gamma_j}
=
\#\{\beta\in\Fq\mid Q(0,\beta,\gamma_j)=0\}.
$$
Since
$$
R(Y,Z)\equiv Q(0,Y,Z)\pmod{Y^q-Y},
$$
$Q(0,Y,\gamma_j)$ and $R_{\gamma_j}(Y)$ have the same values on $\Fq$. Hence
$$
N_{\gamma_j}
=
\#\{\beta\in\Fq\mid R_{\gamma_j}(\beta)=0\}.
$$

\noindent
If $1\le j\le t$, then $a_0(\gamma_j)=0$, so
$$
R_{\gamma_j}(Y)=Y\,S_{\gamma_j}(Y)
$$
with
$$
\deg S_{\gamma_j}\le q-2.
$$
Therefore, the roots of $R_{\gamma_j}(Y)$ in $\Fq$ consist of $0$ together with the nonzero roots of $S_{\gamma_j}(Y)$. By the K\"onig--Rados theorem,
$$
N_{\gamma_j}
=
1+\#\{\beta\in\Fq^\times\mid S_{\gamma_j}(\beta)=0\}
=
1+(q-1)-\rho_{\gamma_j}.
$$
If $t+1\leq j\leq r$, then $a_0(\gamma_j)\neq0$, so $0$ is not a root of $R_{\gamma_j}(Y)$. For $\beta\in\Fq^\times$, we have $\beta^{q-1}=1$, and hence
\begin{align*}
R_{\gamma_j}(\beta)
&=
a_0(\gamma_j)
+a_1(\gamma_j)\beta+\cdots
+a_{q-2}(\gamma_j)\beta^{q-2}
+a_{q-1}(\gamma_j)\beta^{q-1}\\
&=
\bigl(a_0(\gamma_j)+a_{q-1}(\gamma_j)\bigr)
+a_1(\gamma_j)\beta+\cdots
+a_{q-2}(\gamma_j)\beta^{q-2}\\
&=
S_{\gamma_j}(\beta).
\end{align*}
Thus, the nonzero roots of $R_{\gamma_j}(Y)$ are precisely the nonzero roots of $S_{\gamma_j}(Y)$. Since $\deg S_{\gamma_j}\leq q-2$, the K\"onig--Rados theorem gives
$$
N_{\gamma_j}
=
(q-1)-\rho_{\gamma_j}.
$$

\noindent
Therefore
\begin{align*}
    N_0'
&=
\sum_{j=1}^r N_{\gamma_j}\\
&=
\sum_{j=1}^t\bigl(1+(q-1)-\rho_{\gamma_j}\bigr)
+
\sum_{j=t+1}^r\bigl((q-1)-\rho_{\gamma_j}\bigr) \\
&=
r(q-1)-\sum_{j=1}^r\rho_{\gamma_j}+t.
\end{align*}
\end{proof}
\noindent
These general formulas reduce the computation of $N_0'$ to counting the roots of $Q(0,Y,\g)$ in $\Fq$ as $\g$ ranges over the roots of $P(0,Z)$. We now apply these methods to several special forms of the defining polynomials.

\subsection{Homogeneous and one-root reductions}

\noindent
We now study situations in which homogeneity simplifies the computation of $N_0'$. This leads to direct formulas for several families, including cases where the reduced system has a particularly simple form.

\begin{prop}\label{prop:one-root-reduction}
Let $P(X,Z)\in\Fq[X,Z]$ be homogeneous of degree $h\geq2$. Define
$$
r_Q
=
\#\{\be\in\Fq\mid Q(0,\be,0)=0\}.
$$
Then
$
N_0'=r_Q.
$
\end{prop}
\begin{proof}
Since $P(X,Z)$ is homogeneous of degree $h$ and monic in $Z$, we may write
$$
P(X,Z)=Z^{h}+\sum_{i=0}^{h-1} a_iX^{h-i}Z^i
$$
for some $a_i\in\Fq$. Hence $P(0,Z)=Z^{h}.$ Therefore $P(0,\g)=0$ if and only if $\g=0$, and hence
$$
N_0'
=
\#\{\be\in\Fq\mid Q(0,\be,0)=0\}
=
r_Q.
$$
\end{proof}

\begin{theorem}\label{thm:q2-homog}
Let $P(X,Z)\in\Fq[X,Z]$ and $Q(X,Y,Z)\in\Fq[X,Y,Z]$ be homogeneous of degrees $h_1$ and $h_2$ respectively, with $h_1,h_2\ge 2$. Then $N_0'=1.$
\end{theorem}
\begin{proof}
Since $P(X,Z)$ is homogeneous of degree $h_1$ and monic in $Z$, we may write
$$
P(X,Z)=Z^{h_1}+\sum_{i=0}^{h_1-1} a_iX^{h_1-i}Z^i
$$
for some $a_i\in\Fq$, and hence $P(0,Z)=Z^{h_1}.$ Similarly, since $Q(X,Y,Z)$ is homogeneous of degree $h_2$ and monic in $Y$, we have $Q(0,Y,0)=Y^{h_2}.$ Therefore the reduced system
$$
P(0,\g)=0,
\qquad
Q(0,\be,\g)=0
$$
has the unique solution $(\be,\g)=(0,0)$. Thus $N_0'=1.$
\end{proof}

\begin{remark}
By Corollary~\ref{cor:double-total-count}, Theorem~\ref{thm:q2-homog} gives
$$
|S_{d_1,d_2}(\Fq)|=q^2.
$$
This gives the case of square numbers, that is, the $4$-gonal numbers, considered in Section~\ref{S7}.
\end{remark}
\begin{corollary}\label{cor:homogeneous-reduction}
Suppose that $P(0,Z)$ has $r_0$ distinct roots in $\Fq$, of which exactly $r_1$ are nonzero. Let $Q(0,Y,Z)\in\Fq[Y,Z]$ be homogeneous of degree $h$ and monic in $Y$, and define $Q_{Y/Z}(W):=Q(0,W,1).$ Assume that $Q_{Y/Z}(W)$ has exactly $r_2$ distinct roots in $\Fq$. Then $N_0'=r_1r_2+\delta,$ where
$$
\delta=
\begin{cases}
1, & \text{if } 0 \text{ is a root of } P(0,Z),\\
0, & \text{otherwise}.
\end{cases}
$$
\end{corollary}
\begin{proof}
Let $\g_1,\dots,\g_{r_1}\in\Fq^\times$ be the distinct nonzero roots of $P(0,Z)$. Since $Q(0,Y,Z)$ is homogeneous of degree $h$, we have 
$$Q(0,Y,Z)=Z^hQ_{Y/Z}(Y/Z).$$
Fix $i\in\{1,\dots,r_1\}$. Since $\g_i\neq 0$, for $\be\in\Fq$,
$$
Q(0,\be,\g_i)=0
\iff
Q_{Y/Z}(\be/\g_i)=0.
$$ 
Hence the number of $\be\in\Fq$ satisfying $Q(0,\be,\g_i)=0$ is exactly $r_2$. Therefore the total contribution from the nonzero roots of $P(0,Z)$ is $r_1r_2.$

\medskip
\noindent
If $0$ is also a root of $P(0,Z)$, then by Theorem~\ref{thm:q2-homog}, the reduced system at $\g=0$ contributes exactly one solution in $\be$. Hence
$$
N_0'=r_1r_2+\delta.
$$
\end{proof}

\begin{prop}\label{prop:homogeneous-qminus2}
Let $q>3$ and $P(0,Z)$ has $r_0$ distinct roots in $\Fq$, of which exactly $r_1$ are nonzero. Let $Q(0,Y,Z)\in\Fq[Y,Z]$ be homogeneous of degree $q-2$ and monic in $Y$, and define $Q_{Y/Z}(W):=Q(0,W,1).$ Let $\rho_Q$ be the rank of the left circulant matrix associated with $Q_{Y/Z}(W)$, and set
$$
\delta_Q=
\begin{cases}
1, & \text{if } Q(0,0,1)=0,\\
0, & \text{if } Q(0,0,1)\neq 0.
\end{cases}
$$
Then $N_0'=r_1\bigl(q-1-\rho_Q+\delta_Q\bigr)+\delta,$ where
$$
\delta=
\begin{cases}
1, & \text{if } 0 \text{ is a root of } P(0,Z),\\
0, & \text{if } 0 \text{ is not a root of } P(0,Z).
\end{cases}
$$
\end{prop}

\begin{proof}
Since $Q(0,Y,Z)$ is homogeneous of degree $q-2$, we may write
$$
Q(0,Y,Z)=Z^{q-2}Q_{Y/Z}(Y/Z),
\qquad
Q_{Y/Z}(W)=Q(0,W,1).
$$
Let $\g$ be a nonzero root of $P(0,Z)$. Then, for $\be\in\Fq$,
$$
Q(0,\be,\g)=0
\iff
\g^{q-2}Q_{Y/Z}(\be/\g)=0
\iff
Q_{Y/Z}(\be/\g)=0.
$$
Hence each nonzero root $\g$ of $P(0,Z)$ contributes exactly as many values of $\be$ as there are roots of $Q_{Y/Z}(W)$ in $\Fq$.

\noindent
Now $Q_{Y/Z}(W)$ has degree at most $q-2$, so the K\"onig--Rados theorem applies. Therefore the number of nonzero roots of $Q_{Y/Z}(W)$ in $\Fq$ is
$$
q-1-\rho_Q.
$$
Also, $0$ is a root of $Q_{Y/Z}(W)$ if and only if
$$
Q_{Y/Z}(0)=Q(0,0,1)=0.
$$
Thus the total number of roots of $Q_{Y/Z}(W)$ in $\Fq$ is
$$
q-1-\rho_Q+\delta_Q.
$$
Since $P(0,Z)$ has exactly $r_1$ nonzero roots, the total contribution from these roots is
$$
r_1\bigl(q-1-\rho_Q+\delta_Q\bigr).
$$

\noindent
If $0$ is also a root of $P(0,Z)$, then by homogeneity
$$
Q(0,Y,0)=Y^{q-2},
$$
so $\be=0$ is the unique solution of $Q(0,\be,0)=0$. Hence the contribution from $\g=0$ is $\delta$. Therefore
$$
N_0'=r_1\bigl(q-1-\rho_Q+\delta_Q\bigr)+\delta.
$$
\end{proof}

\subsection{Monomial and binomial families}

\noindent
We next study families in which the reduced equation $Q(0,Y,\g)=0$ is monomial or binomial in the variable $Y.$ In these cases, the computation of $N_0'$ is governed by the solvability of power equations over $\Fq,$ and several explicit formulas can be obtained.

\begin{theorem}\label{prop:Ym-aZm}
Suppose that
$
Q(X,Y,Z)=XQ_1(X,Y,Z)+Y^m-b(Z),
$
where
$
Q_1(X,Y,Z)\in\Fq[X,Y,Z],
~
b(Z)\in\Fq[Z],
$
with $\deg_Y Q_1(X,Y,Z)<m$ and $m\geq2$. Set
$d_m=\gcd(m,q-1)$.
\begin{enumerate}
    \item Define
$$
\begin{aligned}
n_0
&=
\#\{\g\in\Fq\mid P(0,\g)=0,\ b(\g)=0\},\\
n_1
&=
\#\left\{\g\in\Fq\mid P(0,\g)=0,\ 
b(\g)\neq0,\ 
b(\g)^{(q-1)/d_m}=1\right\}.
\end{aligned}
$$
Then $N_0'=n_0+d_m\,n_1$. 
\item 
In particular, if $b(Z)=aZ^m$, where $a\in\Fq$, and if $P(0,Z)$ has $r$ distinct roots in $\Fq$, then
$$
N_0'=
\begin{cases}
r, & a=0,\\
\delta, & a\neq 0 \text{ and } a^{(q-1)/d_m}\neq 1,\\
d_mr+(1-d_m)\delta, & a\neq 0 \text{ and } a^{(q-1)/d_m}=1,
\end{cases}
$$
where
$$
\delta=
\begin{cases}
1, & P(0,0)=0,\\
0, & P(0,0)\neq 0.
\end{cases}
$$
\end{enumerate}
\end{theorem}

\begin{proof}  For each root $\g\in \Fq$ of $P(0,Z),$ let $Q(0,Y,\g)$ be the reduced polynomial. 
\begin{enumerate}
\item  For $Q(0,Y,\g)=0$ we get $Y^m=b(\g).$ If $b(\g)=0,$ then $Y=0$ is the unique solution, contributing $n_0$ solutions in total. If $b(\g)\neq 0,$ then by Lemma~\ref{lem:binomial-solvability}, the equation $Y^m=b(\g)$ has solutions in $\Fq$ if and only if $b(\g)^{(q-1)/d_m}=1,$ and in that case it has exactly $d_m$ solutions. The number of such values of $\g$ is $n_1$. Therefore $N_0'=n_0+d_m\,n_1.$
\medskip

\item If $a=0$, then $b(\g)=0$ for every root $\g$ of $P(0,Z)$. Thus $n_0=r$ and $n_1=0$, so $N_0'=r$.

\medskip
\noindent
Now let $a\neq 0$. Then $b(\g)=a\g^m$, so $b(\g)=0$ if and only if $\g=0$. Hence $n_0=\delta$. 

\noindent
For $\g\neq 0$,
$$
b(\g)^{(q-1)/d_m}=(a\g^m)^{(q-1)/d_m}=a^{(q-1)/d_m}.
$$
Therefore, if $a^{(q-1)/d_m}\neq 1$, then $n_1=0$ and $N_0'=\delta$.

\medskip
\noindent
If $a^{(q-1)/d_m}=1$, then exactly the nonzero roots of $P(0,Z)$ contribute to $n_1$, so $n_1=r-\delta$. Hence $N_0'=n_0+d_m\,n_1=\delta+d_m(r-\delta)=d_mr+(1-d_m)\delta.$
\end{enumerate}
\end{proof}

\begin{corollary}{\label{hexa}}
If $P(0,Z)$ has $r$ distinct roots in $\Fq$ and $Q(0,Y,Z)=Y^m$ for some $m\geq2$ is a monomial, then $N_0'=r.$
\end{corollary}

\begin{proof}
The proof follows from Theorem \ref{prop:Ym-aZm} $(ii)$ by taking $a=0$.
\end{proof}

\begin{remark}
In particular, in Corollary \ref{hexa}, if $P(0,Z)=Z^q-Z$, then $r=q$. Hence
$
N_0'=q$ and $
|S_{d_1,d_2}(\Fq)|=2q^2-q.
$
Thus, the point count is the $q$-th hexagonal number, corresponding to the $6$-gonal numbers discussed in Section~\ref{S7}.
\end{remark}
\noindent
The family considered above also admits an expression in terms of multiplicative characters.

\begin{theorem}\label{thm:multchar-binomial}
Suppose that $P(0,Z)$ has $r$ distinct roots in $\Fq$. Let
$$
Q(X,Y,Z)=XQ_1(X,Y,Z)+Y^m-b(Z),
$$
where
$
Q_1(X,Y,Z)\in\Fq[X,Y,Z],
~
b(Z)\in\Fq[Z],
$
with $\deg_Y Q_1(X,Y,Z)<m$ and $m\geq2$. Set
$
d_m=\gcd(m,q-1).
$
Then
$$
N_0'
=
r+
\sum_{\substack{\psi^{d_m}=\psi_0\\ \psi\neq\psi_0}}
\sum_{\substack{\g\in\Fq\\ P(0,\g)=0}}
\psi(b(\g)),
$$
where the outer sum is taken over all nontrivial multiplicative characters $\psi$ of $\Fq$ satisfying $\psi^{d_m}=\psi_0$.
\end{theorem}
\begin{proof}
By Theorem~\ref{thm:rootwise-reduction},
$$
N_0'
=
\sum_{\substack{\g\in\Fq\\ P(0,\g)=0}}
\#\{\be\in\Fq\mid \be^m=b(\g)\}.
$$
The multiplicative-character formula in Lemma~\ref{multcharsol} extends in the same way to $\Fq$, with $d_m=\gcd(m,q-1)$. If $b(\g)\neq0$, it gives
$$
\#\{\be\in\Fq\mid \be^m=b(\g)\}
=
\sum_{\psi^{d_m}=\psi_0}\psi(b(\g)).
$$
If $b(\g)=0$, the equation $\be^m=0$ has the unique solution $\be=0$, and by our convention for multiplicative characters,
$$
\sum_{\psi^{d_m}=\psi_0}\psi(0)=1.
$$
Hence, in all cases,
$$
N_0'
=
\sum_{\substack{\g\in\Fq\\ P(0,\g)=0}}
\sum_{\psi^{d_m}=\psi_0}\psi(b(\g)).
$$
Separating the trivial character $\psi_0$, which contributes $1$ for each of the $r$ roots of $P(0,Z)$, gives
$$
N_0'
=
r+
\sum_{\substack{\psi^{d_m}=\psi_0\\ \psi\neq\psi_0}}
\sum_{\substack{\g\in\Fq\\ P(0,\g)=0}}
\psi(b(\g)).
$$
\end{proof}
\noindent
Now we will look at the case when $Q(0,Y,Z)$ is a binomial.
\begin{theorem}\label{thm:binomial-two-term}
Suppose that $P(0,Z)$ has exactly $r$ distinct roots $\g_1,\ldots,\g_r$ in $\Fq$, and set $I=\{\g_1,\ldots,\g_r\}.$ Let $Q(X,Y,Z)=XQ_1(X,Y,Z)+a(Z)Y^m-b(Z)Y^n,$ where
$
Q_1(X,Y,Z)\in\Fq[X,Y,Z],
$ $
a(Z), ~b(Z)\in\Fq[Z],
$
and $0\leq n<m\leq q-1$.  For $\g\in I$, set
$$
N_\g=\#\{\be\in\Fq\mid Q(0,\be,\g)=0\}.
$$
Then $N_0'=\sum_{\g\in I}N_\g$. Moreover:
\begin{enumerate}
\item If $n>0$, set $d_{m-n}=\gcd(m-n,q-1)$, and
$$
\begin{aligned}
n_{0}&=\#\{\g\in I\mid a(\g)=0,\ b(\g)=0\},\\
n_1&=\#\left\{\g\in I\mid a(\g)b(\g)\neq 0,\ 
\left(\frac{b(\g)}{a(\g)}\right)^{(q-1)/d_{m-n}}=1\right\}.
\end{aligned}
$$
Then
$
N_0'=r+(q-1)n_{0}+d_{m-n}n_1. 
$

\item If $n=0$, set $d_m =\gcd(m,q-1)$, and
$$
\begin{aligned}
n_{0}&=\#\{\g\in I\mid a(\g)=0,\ b(\g)=0\},\\
n_{0}'&=\#\{\g\in I\mid a(\g)\neq 0,\ b(\g)=0\},\\
n_1&=\#\left\{\g\in I\mid a(\g)b(\g)\neq 0,\ 
\left(\frac{b(\g)}{a(\g)}\right)^{(q-1)/d_m}=1\right\}.
\end{aligned}
$$
Then
$
N_0'=qn_{0}+n_{0}'+d_mn_1.
$
\end{enumerate}
\end{theorem}
\begin{proof}
By Theorem~\ref{thm:rootwise-reduction},
$$
N_0'=\sum_{\g\in I}N_\g.
$$
We compute $N_\g$ in the two cases.

\begin{enumerate}
\item Suppose that $n>0$. Then
$$
Q(0,Y,\g)=a(\g)Y^m-b(\g)Y^n
=
Y^n\bigl(a(\g)Y^{m-n}-b(\g)\bigr).
$$ Hence $Y=0$ is always a solution.

\medskip
\noindent
If $a(\g)=b(\g)=0$, then $Q(0,Y,\g)=0$ for every $Y\in\Fq$, so $N_\g=q$. If exactly one of $a(\g)$ and $b(\g)$ is zero, then $Y=0$ is the only solution, so $N_\g=1$.

\medskip
\noindent
Now suppose that $a(\g)b(\g)\neq0$. The nonzero solutions are precisely the solutions of
$$
Y^{m-n}=\frac{b(\g)}{a(\g)}.
$$
Since $d_{m-n}=\gcd(m-n,q-1)$, Lemma~\ref{lem:binomial-solvability} gives exactly $d_{m-n}$ nonzero solutions if
$
\left(\frac{b(\g)}{a(\g)}\right)^{(q-1)/d_{m-n}}=1,
$
and no nonzero solution otherwise. 

\medskip
\noindent
Thus, for $n>0$,
$$
N_\g=
\begin{cases}
q, & a(\g)=0,\ b(\g)=0,\\
1+d_{m-n}, & a(\g)b(\g)\neq 0 \text{ and }
\left(\dfrac{b(\g)}{a(\g)}\right)^{(q-1)/d_{m-n}}=1,\\
1, & \text{otherwise}.
\end{cases}
$$
Summing over $\g\in I$ gives
\begin{align*}
N_0'
&= n_0q+n_1(1+d_{m-n})+(r-n_0-n_1)\\
&=r+(q-1)n_0+d_{m-n}n_1.
\end{align*}

\item Suppose that $n=0$. Then
$$
Q(0,Y,\g)=a(\g)Y^m-b(\g).
$$

\noindent
If $a(\g)=b(\g)=0$, then the equation is identically zero, so $N_\g=q$. If $a(\g)\neq 0$ and $b(\g)=0$, then $Y=0$ is the unique solution, so $N_\g=1$. If $a(\g)=0$ and $b(\g)\neq 0$, then the equation has no solution.

\medskip
\noindent
Now suppose that $a(\g)b(\g)\neq 0$. Then the solutions are precisely the solutions of $Y^m=\frac{b(\g)}{a(\g)}.$ Since $d_m=\gcd(m,q-1)$, Lemma~\ref{lem:binomial-solvability} gives exactly $d_m$ nonzero solutions if $\left(\frac{b(\g)}{a(\g)}\right)^{(q-1)/d_m}=1,$ and no solution otherwise.

\medskip
\noindent
Thus, for $n=0$,
$$
N_\g=
\begin{cases}
q, & a(\g)=0,\ b(\g)=0,\\
1, & a(\g)\neq 0,\ b(\g)=0,\\
d_m, & a(\g)b(\g)\neq 0 \text{ and }
\left(\dfrac{b(\g)}{a(\g)}\right)^{(q-1)/d_m}=1,\\
0, & \text{otherwise}.
\end{cases}
$$
\noindent
Summing over $\g\in I$ gives
$$
N_0'
=
qn_0+n_0'+d_mn_1.
$$
\end{enumerate}
\end{proof}
\noindent Taking $a(Z)=1$ in Theorem~\ref{thm:binomial-two-term} gives the following.

\begin{corollary}\label{cor:binomial-special}
Suppose that $P(0,Z)$ has exactly $r$ distinct roots $\g_1,\ldots,\g_r$ in $\Fq$, and set $I=\{\g_1,\ldots,\g_r\}$. Let $Q(0,Y,Z)=Y^m-b(Z)Y^n,$ where $0\le n<m\le q-1$ and $b(Z)\in\Fq[Z]$. Thus, for every $\g\in I$, the polynomial $Q(0,Y,\g)=Y^m-b(\g)Y^n$ is a binomial in $Y$.  Then the following hold.

\begin{enumerate}
\item If $n>0$, set $d_{m-n}=\gcd(m-n,q-1)$ and $n_1=\#\{\g\in I\mid b(\g)\neq0,\ b(\g)^{(q-1)/d_{m-n}}=1\}.$ Then $N_0'=r+d_{m-n}n_1$.

\item If $n=0$, set $d_m=\gcd(m,q-1)$ and
\begin{align*}
n_0'&=\#\{\g\in I\mid b(\g)=0\},\\
n_1&=\#\{\g\in I\mid b(\g)\neq0,\ b(\g)^{(q-1)/d_m}=1\}.
\end{align*}
Then $N_0'=n_0'+d_mn_1$.
\end{enumerate}
\end{corollary}

\begin{remark}
If
$
P(0,Z)=Z^q-Z,
~
Q(0,Y,Z)=Y^2-ZY,
$
then Corollary~\ref{cor:binomial-special} gives
$
N_0'=2q-1.
$
Hence
$
|S_{d_1,d_2}(\Fq)|=3q^2-2q,
$
which is the $q$-th octagonal number, corresponding to the $8$-gonal numbers discussed in Section~\ref{S7}.
\end{remark} 
\subsection{Special permutation-polynomial families}

\noindent
\noindent
We next consider special families for which $Q(0,Y,\g)$ is a permutation
polynomial of $\Fq$ for every root $\g$ of $P(0,Z)$.

\begin{defi}
A polynomial $f\in\Fq[X]$ is called a permutation polynomial of $\Fq$ if the map
$$
f:\Fq\to\Fq,\qquad c\mapsto f(c),
$$
is a permutation of $\Fq$.
\end{defi}

\begin{lemma}\label{lem:perm-poly-criteria}
Let $a,b,c\in\Fq.$ The following polynomials are permutation polynomials of $\Fq$ under the stated hypotheses.
$$
\renewcommand{\arraystretch}{1.25}
\begin{array}{|c|c|c|}
\hline
\text{Case} & f(Y) & \text{Hypotheses} \\
\hline
(1)\ \text{\cite[Proposition 2.7]{AFF}}
&
Y^i+aY^j+b
&
\begin{array}{c}
i>j\ge 1,\ \gcd(i-j,q-1)=1,\\
a=0,\ \gcd(i,q-1)=1
\end{array}
\\
\hline
(2)\ \text{\cite[Corollary 2.10]{AFF}}
&
Y^3+aY^2+bY+c
&
\begin{array}{c}
\operatorname{char}(\Fq)\neq 3,\\
a^2=3b,\ q\equiv 2 \pmod 3
\end{array}
\\
\hline
(3)\ \text{\cite[Corollary 2.14]{AFF}}
&
Y^3+aY^2+bY+c
&
\begin{array}{c}
\operatorname{char}(\Fq)=3,\\
a=0,\ \eta(-b)\neq 1
\end{array}
\\
\hline
(4)\ \text{\cite[Theorem 2.15]{AFF}}
&
Y^i+aY+b
&
\begin{array}{c}
i>1,\ i\text{ is not a power of }\operatorname{char}(\Fq),\\
q\ge (i^2-4i+6)^2,\ a=0,\ \\ \gcd(i,q-1)=1
\end{array}
\\
\hline
(5)\ \text{\cite[Exercise 2.4]{AFF}}
&
Y^{p^s}-aY^{p^r}
&
\begin{array}{c}
q=p^n,\ s>r\ge 0,\\
a\text{ is not a }(p^s-p^r)\text{-th power in }\Fq,\\
\text{equivalently, } \\ a\text{ is not a }(p^{s-r}-1)\text{-th power in }\Fq
\end{array}
\\
\hline
\end{array}
$$
Here $\eta$ denotes the quadratic character of $\Fq$.
\end{lemma}

\begin{theorem}\label{thm:permcriterion}
Suppose that $P(0,Z)$ has exactly $r$ distinct roots in $\Fq$. Assume that $Q(0,Y,Z) \allowbreak =f(Y)\in\Fq[Y],$ and that $f(Y)$ satisfies one of the cases $(1)$--$(5)$ of Lemma~\ref{lem:perm-poly-criteria}. Then $N_0'=r.$
\end{theorem}
\begin{proof}
Since $Q(0,Y,Z)=f(Y)\in\Fq[Y]$, for every $\g\in\Fq$ satisfying $P(0,\g)=0$, we have $Q(0,Y,\g)=f(Y).$ By hypothesis and Lemma~\ref{lem:perm-poly-criteria}, $f(Y)$ is a permutation polynomial of $\Fq$. Hence the equation $Q(0,Y,\g)=0$ has exactly one solution in $\Fq$ for every root $\g$ of $P(0,Z)$. Since $P(0,Z)$ has exactly $r$ distinct roots in $\Fq$, Proposition~\ref{thm:rootwise-reduction} gives $N_0'=r.$
\end{proof}

\begin{lemma}[\cite{LM13}, Theorem 7.7]\label{permutationzero}
A polynomial $f\in\Fq[X]$ is a permutation polynomial of $\Fq$ if and only if
$$
\sum_{a\in\Fq}\chi(f(a))=0
$$
for every nontrivial additive character $\chi$ of $\Fq$.
\end{lemma}

\begin{prop}\label{prop:power-and-shifted-power}
Suppose that $P(0,Z)$ has exactly $r$ distinct roots in $\Fq$, and let
$m\geq2$ satisfy $\gcd(m,q-1)=1$. Let
$Q_1(X,Y,Z)\in\Fq[X,Y,Z]$ and $b(Z)\in\Fq[Z]$, with
$\deg_Y Q_1(X,Y,Z)<m$. Then $N_0'=r$ in each of the following cases:
\begin{enumerate}
\item
$Q(X,Y,Z)=XQ_1(X,Y,Z)+Y^m+b(Z),$
\item
$Q(X,Y,Z)=XQ_1(X,Y,Z)+(Y+b(Z))^m.$
\end{enumerate}
\end{prop}

\begin{proof}
By Theorem \ref{thm:additive-master-double}, in case \textup{(i)},
$$
N_0'
=
r+\frac{1}{q}
\sum_{s\in \mathbb F_q^\times}
\sum_{\substack{\gamma\in \mathbb F_q\\ P(0,\gamma)=0}}
\chi_1(sb(\gamma))
\sum_{\beta\in \mathbb F_q}\chi_1(s\beta^m).
$$
Since $\gcd(m,q-1)=1$, the polynomial $Y^m$ is a permutation polynomial of $\mathbb F_q$. For each $s\in\Fq^\times$, the map
$a\mapsto\chi_1(sa)$ is a nontrivial additive character of $\Fq$.
 Hence, by Lemma \ref{permutationzero},
$$
\sum_{\beta\in \mathbb F_q}\chi_1(s\beta^m)=0
\qquad
\text{for every } s\in \mathbb F_q^\times .
$$
Therefore, the sum vanishes and we get $N_0'=r$.
\noindent
In case \textup{(ii)}, Theorem \ref{thm:additive-master-double} gives
$$
N_0'
=
r+\frac{1}{q}
\sum_{s\in \mathbb F_q^\times}
\sum_{\substack{\gamma\in \mathbb F_q\\ P(0,\gamma)=0}}
\sum_{\beta\in \mathbb F_q}
\chi_1\bigl(s(\beta+b(\gamma))^m\bigr).
$$
For fixed $\gamma \in \Fq$ such that $P(0,\gamma)=0,$  the substitution $u=\beta+b(\gamma)$ reduces the inner sum to
$$
\sum_{u\in \mathbb F_q}\chi_1(su^m),
$$
which is zero by the same argument. Hence $N_0'=r$.
\end{proof}

\subsection{Quadratic and character-theoretic families}

\noindent
We next consider the case where $Q(0,Y,\g)$ is quadratic in $Y$ and obtain a formula for $N_0'$ using character sums and Gauss sums.

\medskip
\noindent
We first recall the following properties of Gauss sums.

\begin{lemma}[\cite{LM13}, Theorem 5.12]\label{lem:gauss-properties}
Let $\psi$ be a multiplicative character and $\chi$ an additive character of $\Fq$. Then the Gauss sum
$$
G(\psi,\chi)=\sum_{c\in\Fq}\psi(c)\chi(c)
$$
satisfies:
\begin{enumerate}
\item $G(\psi,\chi_{ab})=\overline{\psi(a)}\,G(\psi,\chi_b)$ for $a\in\Fq^\times$ and $b\in\Fq$.
\item $G(\overline{\psi},\chi)=\psi(-1)\,\overline{G(\psi,\chi)}.$
\item $G(\psi,\chi)\,G(\overline{\psi},\chi)=\psi(-1)\,q$ for $\psi\neq\psi_0$ and $\chi\neq\chi_0$.
\end{enumerate}
\end{lemma}

\begin{remark}\label{rem:quadratic-character-prime}
Following \cite[Chapter 5, Theorem 1]{IrelandRosen1990} and \cite[Exercise 5.2]{IrelandRosen1990}, we recall the following elementary properties of the quadratic character over the prime field $\FF_p$.
\begin{itemize}
\item If $p$ is an odd prime, then $\eta(-1)=(-1)^{(p-1)/2}.$
\item For every $a\in \FF_p$, the equation $x^2=a$ has exactly $1+\eta(a)$ solutions in $\FF_p$.
\end{itemize}
\end{remark}

\begin{theorem}\label{thm:quadratic-N0}
Assume that $q=p$ is an odd prime, and suppose that $P(0,Z)$ has
exactly $r$ distinct roots in $\Fq$. Let
$
Q(X,Y,Z)=XQ_1(X,Y,Z)+Y^2+b(Z)Y+c(Z),
$
where $Q_1(X,Y,Z)\in\Fq[X,Y,Z]$, $b(Z),c(Z)\in\Fq[Z]$, and
$\deg_Y Q_1(X,Y,Z)<2$. For each $\g\in\Fq$, set $D(\g)=c(\g)-\frac{b(\g)^2}{4}.$
Then
$$
N_0'
=
r+\eta(-1)
\sum_{\substack{\g\in\Fq\\ P(0,\g)=0}}
\eta(D(\g)).
$$
\end{theorem}
\begin{proof}
By Theorem~\ref{thm:additive-master-double},
$$
N_0'
=
r+\frac{1}{p}
\sum_{s\in\Fq^\times}
\sum_{\substack{\g\in\Fq\\ P(0,\g)=0}}
\sum_{\be\in\Fq}\chi_1(sQ(0,\be,\g)).
$$
For each root $\g$ of $P(0,Z)$,
$$
Q(0,\be,\g)
=
\be^2+b(\g)\be+c(\g)
=
\left(\be+\frac{b(\g)}{2}\right)^2+D(\g).
$$
Thus, after the change of variable $u=\be+b(\g)/2$,
$$
\sum_{\be\in\Fq}\chi_1(sQ(0,\be,\g))
=
\chi_1(sD(\g))\sum_{u\in\Fq}\chi_1(su^2).
$$
Since $q=p$ is an odd prime, Remark~\ref{rem:quadratic-character-prime} gives
$$
\#\{u\in\Fq\mid su^2=t\}=1+\eta(s^{-1}t).
$$
Hence
\begin{align*}
\sum_{u\in\Fq}\chi_1(su^2)
&=\sum_{t\in\Fq}\bigl(1+\eta(s^{-1}t)\bigr)\chi_1(t)\\
&=\eta(s)\sum_{t\in\Fq}\eta(t)\chi_1(t)
=\eta(s)G(\eta,\chi_1),
\end{align*}
because $\sum_{t\in\Fq}\chi_1(t)=0$. Therefore
$$
N_0'
=
r+\frac{G(\eta,\chi_1)}{p}
\sum_{\substack{\g\in\Fq\\ P(0,\g)=0}}
\sum_{s\in\Fq^\times}\eta(s)\chi_1(sD(\g)).
$$
For the inner sum, if $D(\g)=0$, then it is $\sum_{s\in\Fq^\times}\eta(s)=0$. If $D(\g)\neq 0$, putting $t=sD(\g)$ gives
$$
\sum_{s\in\Fq^\times}\eta(s)\chi_1(sD(\g))
=
\eta(D(\g))G(\eta,\chi_1).
$$
Since $\eta(0)=0$, this formula holds uniformly for all $\g$. Hence
$$
N_0'
=
r+\frac{G(\eta,\chi_1)^2}{p}
\sum_{\substack{\g\in\Fq\\ P(0,\g)=0}}\eta(D(\g)).
$$
Finally, Lemma~\ref{lem:gauss-properties}(iii) gives $G(\eta,\chi_1)^2=\eta(-1)p$. Therefore
$$
N_0'
=
r+\eta(-1)
\sum_{\substack{\g\in\Fq\\ P(0,\g)=0}}\eta(D(\g)).
$$
\end{proof}

\begin{corollary}\label{cor:quadratic-square-shift}
Assume that $q=p$ is an odd prime. Suppose that $P(0,Z)$ has exactly $r$ distinct roots in $\Fq$ and
$$
Q(X,Y,Z)=XQ_1(X,Y,Z)+(Y+b(Z))^2,
\qquad \deg_Y Q_1(X,Y,Z)<2.
$$
Then $N_0'=r$.
\end{corollary}
\begin{proof}
Since
$$
Q(0,Y,Z)=(Y+b(Z))^2=Y^2+2b(Z)Y+b(Z)^2,
$$
we get, for every $\g\in\Fq$,
$$
D(\g)=b(\g)^2-\frac{(2b(\g))^2}{4}=0.
$$
Thus $\eta(D(\g))=0$ for all $\g\in\Fq$. By Theorem~\ref{thm:quadratic-N0},
$$
N_0'
=
r+\eta(-1)
\sum_{\substack{\g\in\Fq\\ P(0,\g)=0}}\eta(D(\g))
=
r.
$$
\end{proof}

\subsection{\texorpdfstring{Degree $q-2$}{Degree q-2} and circulant families}

\noindent
We next consider families in which the reduced polynomials have degree
$q-2$. In this case, the K\"onig--Rados theorem gives formulas for
$N_0'$ in terms of ranks of left circulant matrices.

\begin{prop}\label{prop:degree-qminus2-product}
Assume that $q\geq4$. Let
$$
\begin{aligned}
P(X,Z)&=XP_1(X,Z)+a_0+a_1Z+\cdots+a_{q-3}Z^{q-3}+Z^{q-2},\\
Q(X,Y,Z)&=XQ_1(X,Y,Z)+b_0+b_1Y+\cdots+b_{q-3}Y^{q-3}+Y^{q-2},
\end{aligned}
$$
where $a_i,b_j\in\Fq$ for all $i,j$, $\deg_Z P_1(X,Z)<q-2,$ and $\deg_Y Q_1(X,Y,Z)<q-2.$ Let $\rho_P$ and $\rho_Q$ be the ranks of the left circulant matrices associated with $P(0,Z)$ and $Q(0,Y,0)$, respectively. Then
$$
N_0'
=
(q-1-\rho_P+\delta_P)(q-1-\rho_Q+\delta_Q),
$$
where
$$
\delta_P=
\begin{cases}
1, & a_0=0,\\
0, & a_0\neq 0,
\end{cases}
\qquad
\delta_Q=
\begin{cases}
1, & b_0=0,\\
0, & b_0\neq 0.
\end{cases}
$$
\end{prop}
\begin{proof}
Since
$
Q(0,Y,Z)=b_0+b_1Y+\cdots+b_{q-3}Y^{q-3}+Y^{q-2},
$
we have
$
Q(0,\be,\g)=Q(0,\be,0)
$
for all $\be,\g\in\Fq$. Hence
$$
N_0'
=
\#\{\g\in\Fq\mid P(0,\g)=0\} \cdot
\#\{\be\in\Fq\mid Q(0,\be,0)=0\}.
$$
By the K\"onig--Rados theorem (Theorem~\ref{kongrados}), the number of nonzero roots of $P(0,Z)$ in $\Fq$ is $q-1-\rho_P$. Also, $0$ is a root of $P(0,Z)$ if and only if $a_0=0$. 

\medskip
\noindent
Thus,
$
\#\{\g\in\Fq\mid P(0,\g)=0\}=q-1-\rho_P+\delta_P.
$

\medskip
\noindent
Similarly,
$
\#\{\be\in\Fq\mid Q(0,\be,0)=0\}=q-1-\rho_Q+\delta_Q.
$

\medskip
\noindent
Therefore, $N_0'=(q-1-\rho_P+\delta_P)(q-1-\rho_Q+\delta_Q).$
\end{proof}
\section{Bounds}\label{bounds}{\label{S5}}

\noindent
In the preceding section, we obtained exact formulas for $N_0'$ in several special cases. We now derive general bounds for $N_0'$.

\begin{prop}
Let $P(0,Z)\in\Fq[Z]$ have degree $m,$ and let
$Q(0,Y,Z)\in\Fq[Y,Z]$ be monic of degree $n$ in $Y.$ Then $N_0'\le mn.$
\end{prop}

\begin{proof}
Let $I:=\{\g\in\Fq\mid P(0,\g)=0\}.$ Since $P(0,Z)$ has degree $m$, we have $|I|\le m$. For each fixed $\g\in I,$ the polynomial $Q(0,Y,\g)$ has degree $n$ in $Y$, and hence has at most $n$ roots in $\Fq.$ Therefore,
\begin{align*}
N_0'
&=\sum_{\g\in I}\#\{\be\in\Fq\mid Q(0,\be,\g)=0\}\\
&\le \sum_{\g\in I} n\\
&=|I|\,n\\
&\le mn.
\end{align*}
\end{proof}

\begin{corollary}
Let $r_P:=\deg\bigl(\gcd(P(0,Z),Z^q-Z)\bigr),$ and, for each $\g\in\Fq$ satisfying $P(0,\g)=0$, let $r_Q(\g):=\deg\bigl(\gcd(Q(0,Y,\g),Y^q-Y)\bigr).$ Set
$$
r_Q^{\max}:=\max_{\substack{\g\in\Fq\\ P(0,\g)=0}} r_Q(\g).
$$
Then $r_P$ is the number of distinct roots of $P(0,Z)$ in $\Fq,$ and for each such $\g,$ the number $r_Q(\g)$ is the number of distinct roots of $Q(0,Y,\g)$ in $\Fq.$ Moreover,
$$
N_0' \le r_P\,r_Q^{\max}.
$$
\end{corollary}

\begin{proof}
The number of roots of $P(0,Z)$ in $\Fq$ is $r_P,$ and for each such root $\g,$ the number of $\be\in\Fq$ satisfying $Q(0,\be,\g)=0$ is $r_Q(\g).$ Therefore
$$
N_0'
=
\sum_{\substack{\g\in\Fq\\ P(0,\g)=0}} r_Q(\g)
\le
\sum_{\substack{\g\in\Fq\\ P(0,\g)=0}} r_Q^{\max}
=
r_P\,r_Q^{\max}.
$$
\end{proof}

\noindent
After the degree and gcd bounds, we next derive bounds for $N_0'$ using additive character sums. These bounds apply when suitable estimates for the corresponding exponential sums are available.

\medskip
\noindent
We recall the following standard estimates for additive character sums over finite fields.

\begin{lemma}[\cite{LM13}, Weil's theorem]
Let $f\in\Fq[X]$ be of degree $m\ge 1$ with $\gcd(m,q)=1,$ and let $\chi$ be a nontrivial additive character of $\Fq.$ Then
$$
\left|\sum_{\al\in\Fq}\chi\bigl(f(\al)\bigr)\right|
\le (m-1)q^{1/2}.
$$
\end{lemma}

\begin{lemma}[\cite{LM13}]\label{thm:binomial-weil}
Let $\chi$ be a nontrivial additive character of $\Fq,$ let $m\in\N,$ and let $d=\gcd(m,q-1).$ Then
$$
\left|\sum_{\al\in\Fq}\chi(a\al^m+b)\right|
\le (d-1)q^{1/2}
$$
for any $a,b\in\Fq$ with $a\neq 0.$
\end{lemma}

\noindent
We now apply the binomial estimate to a special form of $Q$, where the equation obtained after setting $X=0$ is binomial in $Y$.
\begin{prop}
Suppose that $P(0,Z)$ has $r$ distinct roots in $\Fq.$ If $Q(X,Y,Z)=XQ_1(X,Y,Z) \allowbreak +Y^m+b(Z),$ where $Q_1(X,Y,Z)\in\Fq[X,Y,Z],
~
b(Z)\in\Fq[Z],$ $n\in\N$ and $d=\gcd(m,q-1),$ then
$$
|N_0'-r|\le r(d-1)\sqrt q.
$$
In particular,
$$
N_0'=r+O\!\bigl(r(d-1)\sqrt q\bigr).
$$
\end{prop}

\begin{proof}
By Theorem~\ref{thm:additive-master-double},
$$
N_0'-r
=
\frac{1}{q}
\sum_{\substack{\g\in\Fq\\ P(0,\g)=0}}
\sum_{s\in\Fq^\times}
\sum_{\be\in\Fq}
\chi_1\bigl(s(\be^m+b(\g))\bigr).
$$
For each fixed $s\in\Fq^\times$ and $\g\in\Fq$ with $P(0,\g)=0,$ Lemma~\ref{thm:binomial-weil} applied with $a=s$ and $b=sb(\g)$ gives
$$
\left|
\sum_{\be\in\Fq}\chi_1\bigl(s(\be^m+b(\g))\bigr)
\right|
\le (d-1)\sqrt q.
$$
Hence, by the triangle inequality,
$$
|N_0'-r|
\le
\frac{1}{q}\cdot r\cdot(q-1)\cdot(d-1)\sqrt q
\le
r(d-1)\sqrt q.
$$
The $O$-estimate follows immediately.
\end{proof}

\noindent
\noindent
For $q=2^n$, we use the following additive-character estimate.

\begin{lemma}[\cite{Gangopadhyay2003}] \label{2^n}
For any nontrivial additive character $\chi$ of $\mathbb{F}_{2^n}$,
$$
\left|\sum_{\al\in \mathbb{F}_{2^n}} \chi\bigl(g(\al)\bigr)\right|
\le 2^{\,n-t+1},
$$
where
$$
g(X)=X^{c}+\sum_{i=1}^{N} b_i X^{a_i(2^{t}-1)}, 
$$
$c$ is coprime to $2^{n}-1$, $t\mid n$, $N$ is a positive integer, $b_i\in \mathbb{F}_{2^n}$ and $a_i$ is a positive integer for all $i=1,\dots,N$.
\end{lemma}

\noindent
We now apply this estimate to the polynomial obtained after setting $X=0$.

\begin{theorem}
Let $P(0,Z) \in \Fq[Z]$ have $r$ distinct roots in $\Fq$ and $q = 2^n.$ Suppose $t \mid n$ and
$$
Q(X,Y,Z) =XQ_1(X,Y,Z)+ Y^c + \sum_{i=1}^{N} b_i Y^{a_i(2^t-1)} \in \Fq[X,Y,Z],
$$
where $b_i \in \Fq$, $a_i$ are positive integers, and $\gcd(c,2^n-1)=1.$ Then
$$
|N_0' - r| \le r\,2^{\,1-t}(2^{n}-1).
$$
\end{theorem}

\begin{proof}
From Theorem \ref{thm:additive-master-double} and Lemma \ref{2^n}, we get
$$|N_0'-r| \le \frac{1}{q}\, r (q-1) 2^{\,n-t+1}.$$

\noindent 
Substituting $q=2^n$ gives the required bound.
\end{proof}

\noindent
We now sharpen the preceding binomial bound by returning to the exact formula for $N_0'$ in this case.

\begin{theorem}\label{drbound}
Assume that $P(0,Z)\in\Fq[Z]$ has $r$ distinct roots in $\Fq$. Let $Q(X,Y,Z)=XQ_1(X,Y,Z)+Y^m-f(Z),$ where $Q_1(X,Y,Z)\in\Fq[X,Y,Z]$, $f(Z)\in\Fq[Z]$,
$m\geq2$, and $\deg_YQ_1<m$. Let $d=\gcd(m,q-1).$ Define $n_0=\#\{\g\in\Fq\mid P(0,\g)=0,\ f(\g)=0\}.$ Then
$$
n_0\le N_0'\le dr-(d-1)n_0\le dr.
$$
\end{theorem}
\begin{proof}
Let $n_1=\#\{\g\in\Fq\mid P(0,\g)=0,\ f(\g)^{(q-1)/d}=1\}.$
Since $P(0,Z)$ has $r$ distinct roots in $\Fq$, we have
$$
n_0+n_1\le r.
$$
Thus $n_1\le r-n_0.$ By Theorem~\ref{prop:Ym-aZm}, we have
$$
N_0'=n_0+dn_1.
$$
Since $n_1\ge 0$, this gives $N_0'\ge n_0.$ On the other hand, using $n_1\le r-n_0$, we get
\begin{align*}
N_0'
&=
n_0+dn_1\\
&\le
n_0+d(r-n_0)\\
&=
dr-(d-1)n_0.
\end{align*}
Since $n_0\ge 0$, we also have $dr-(d-1)n_0\le dr.$ Therefore
$$
n_0\le N_0'\le dr-(d-1)n_0\le dr.
$$
\end{proof}

\begin{prop}{\label{prop:upper-bound-attained}}
Let $Q(X,Y,Z)$, $r$, and $d$ be as in Theorem~\ref{drbound}. Suppose that for every root $\g$ of $P(0,Z)$ in $\Fq$, $f(\g)\in(\Fq^\times)^m.$ Then $N_0'=dr$, so the upper bound in Theorem~\ref{drbound} is attained. \end{prop}
\begin{proof}
Since $f(\g)\in(\Fq^\times)^m$ for every root $\g$ of $P(0,Z)$ in $\Fq$, we have $f(\g)\neq 0$ and $f(\g)^{(q-1)/d}=1.$ Hence $n_0=0$ and $n_1=r.$ Further, by Theorem~\ref{prop:Ym-aZm}, $N_0'=n_0+dn_1.$ Therefore $N_0'=dr.$
\end{proof}

\begin{ex}
Consider the field $\F_{17}$. Take $P(0,Z)=Z^8-1$ and $Q(X,Y,Z)=XQ_1(X,Y,Z)+Y^6-Z,$ where $Q_1(X,Y,Z)\in\F_{17}[X,Y,Z]$ and $\deg_YQ_1<6$. Thus $m=6$ and $f(Z)=Z$. The roots of $P(0,Z)$ in $\F_{17}$ are
$$
I=\{1,2,4,8,9,13,15,16\},
$$
so $r=8$. Moreover, $d=\gcd(6,16)=2.$ Indeed, the following table shows that every element of $I$ is a sixth power in $\F_{17}^{\times}$:
\begin{table}[h]
\centering
\setlength{\tabcolsep}{12pt}
\renewcommand{\arraystretch}{1.3}
\begin{tabular}{|c|c|c|c|c|c|c|c|}
\hline
$1$ & $2$ & $4$ & $8$ & $9$ & $13$ & $15$ & $16$ \\
\hline
$1^6$ & $5^6$ & $8^6$ & $6^6$ & $7^6$ & $2^6$ & $3^6$ & $4^6$ \\
\hline
\end{tabular}
\end{table}

\noindent
Since $f(Z)=Z$, it follows that $f(\g)=\g\in(\F_{17}^{\times})^6$ for every $\g\in I$. Hence, by Proposition~\ref{prop:upper-bound-attained}, $N_0'=dr=2\cdot 8=16.$ Thus the upper bound in Theorem~\ref{drbound} is attained.
\end{ex}

\section{A Macaulay2 Algorithm}{\label{S6}}
\noindent\textbf{Algorithm:}
Let $q$ be a prime power and let $R=\mathbb F_q[X,Y,Z,T]$ be the polynomial ring over the finite field $\mathbb F_q$. Let $P(X,Z)\in \mathbb F_q[X,Z]$ be monic in $Z$, and let $Q(X,Y,Z)\in \mathbb F_q[X,Y,Z]$ be monic in $Y$. Consider the double Danielewski surface $S_{d_1,d_2}\subseteq \mathbb A^4_{\mathbb F_q}$ defined by
$$
X^{d_1}Y=P(X,Z),
\qquad
X^{d_2}T=Q(X,Y,Z).
$$
\noindent
To count the $\mathbb F_q$-rational points of $S_{d_1,d_2}$, we consider the ideal
$$
I=
\left(
X^{d_1}Y-P(X,Z),
X^{d_2}T-Q(X,Y,Z),
X^q-X,
Y^q-Y,
Z^q-Z,
T^q-T
\right)
\subseteq R.
$$

\noindent\textbf{Input:}

\begin{enumerate}
\item Define the finite field $K=\mathbb F_q$ and the polynomial ring $R=K[X,Y,Z,T].$ We will consider the lexicographic monomial ordering $Y>Z>X>T.$

\item Define the polynomials $P(X,Z)$ and $Q(X,Y,Z)$ whose $\Fq$-rational points we want to compute.

\item For fixed positive integers $d_1$ and $d_2,$ define the corresponding ideal
$$
I=
\left(
X^{d_1}Y-P(X,Z),
X^{d_2}T-Q(X,Y,Z),
X^q-X,
Y^q-Y,
Z^q-Z,
T^q-T
\right).
$$

\item Compute a Gr\"obner basis $J$ of the ideal $I$.

\item Compute the dimension of the quotient ring $R/( J).$
\end{enumerate}

\noindent\textbf{Output:}
The dimension $\dim_{\mathbb F_q}\left(R/( J)\right),$ is equal to the $\Fq$-rational points of $S_{d_1,d_2}(\mathbb F_q).$ Thus the algorithm returns the number of $\mathbb F_q$-rational points on the double Danielewski surface $S_{d_1,d_2}$.

\medskip

\noindent\textbf{Example:}
Consider the double Danielewski surface $S_{2,3}\subseteq \mathbb A^4_{\mathbb F_q}.$ We consider the polynomials $P(X,Z)=Z^5+X^4Z+Z^3X^2$ and $Q(X,Y,Z)=Y^7+Y^5ZX+Y^2XZ^4+Z^4X^3.$ Thus the corresponding ideal is
$$
I=
\left(
X^2Y-P(X,Z),
X^3T-Q(X,Y,Z),
X^q-X,
Y^q-Y,
Z^q-Z,
T^q-T
\right)
\subseteq \mathbb F_q[X,Y,Z,T].
$$
\noindent
By Theorem \ref{thm:q2-homog}, the number of $\mathbb F_q$-rational points on this surface is $q^2.$ We verify this computationally using Macaulay2 \cite{M2}.

\begin{verbatim}
i1 : dDanSurface = q -> (
         K = GF(q);
         R = K[y,z,x,t, MonomialOrder => Lex];
         P = z^5 + x^4*z + z^3*x^2;
         Q = y^7 + y^5*z*x + y^2*x*z^4 + z^4*x^3;
         I = ideal(x^2*y - P, x^3*t - Q, x^q - x, y^q - y, z^q - z, t^q - t);
         J = groebnerBasis I;
         (q, #flatten entries basis(R/ideal J))
     )

o1 = dDanSurface
o1 : FunctionClosure
i2 : primePowers = {2,3,4,5,7}
o2 = {2, 3, 4, 5, 7}
o2 : List
i3 : results = apply(primePowers, q -> dDanSurface(q))
o3 = {(2, 4), (3, 9), (4, 16), (5, 25), (7, 49)}
o3 : List
i4 : dDanSurface(9)
o4 = (9, 81)
o4 : Sequence
\end{verbatim}
The output is of the form
$
(q,\#S_{2,3}(\mathbb F_q)),
$
where $q$ denotes the finite field $\mathbb F_q.$ For the prime powers $q=2,3,4,5,7,9,$ the computation gives $
(2,4),~
(3,9),~
(4,16),~
(5,25),~
(7,49)$ and $
(9,81).
$ Thus, in each case $\#S_{2,3}(\mathbb F_q)=q^2.$

\section{Polygonal numbers}\label{S7}
\noindent
We now show that, for suitable choices of the polynomials $P(X,Z)$ and $Q(X,Y,Z)$, the resulting surfaces generate the sequence of polygonal numbers.

\begin{defi}
Let $m\ge 3$ be an integer.  A polygonal number associated with an $m$-gon is called an $m$-gonal number. The $n$-th $m$-gonal number, denoted by $S_m(n),$ is defined by
$$
S_m(n)=\sum_{k=0}^{n-1}\bigl(1+k(m-2)\bigr).
$$
Equivalently,
$$
S_m(n)=\frac{n}{2}\bigl((m-2)n-(m-4)\bigr).
$$
\end{defi}

\noindent
Thus $S_4(n)=n^2,$ $S_6(n)=2n^2-n,$ $S_8(n)=3n^2-2n,$ and $S_{10}(n)=4n^2-3n$ give the square, hexagonal, octagonal, and decagonal numbers, respectively. For $n=q=p^r,$ these formulas give prime-power-indexed subsequences.

\medskip
\noindent
Further, the following theorem realizes the prime-power-indexed square, hexagonal, octagonal, and decagonal sequences as $\Fq$-rational point counts of suitabledouble Danielewski surfaces.

\begin{theorem}\label{thm:polygonal-examples}
Let $q=p^n$ and consider the double Danielewski surface
$$
S_{d_1,d_2}
=
\{(\al,\be,\g,\ka)\in\A^4_{\Fq}\mid
\al^{d_1}\be=P(\al,\g),~
\al^{d_2}\ka=Q(\al,\be,\g)\}.
$$
Then, for the following choices of $P(0,Z)$ and $Q(0,Y,Z),$ the number of $\Fq$-rational points on $S_{d_1,d_2}$ is as follows:
\begin{center}
\renewcommand{\arraystretch}{1.35}
\begin{tabular}{|c|c|c|c|c|}
\hline
No.
&
$P(0,Z)$
&
$Q(0,Y,Z)$
&
$\#S_{d_1,d_2}(\Fq)$
&
Polygonal type
\\
\hline
1
&
$Z^2$
&
$Y^2$
&
$q^2$
&
Square
\\
\hline
2
&
$Z^q-Z$
&
$Y^2$
&
$2q^2-q$
&
Hexagonal
\\
\hline
3
&
$Z^q-Z$
&
$Y^2-ZY$
&
$3q^2-2q$
&
Octagonal
\\
\hline
4
&
$Z^q-Z$
&
$Y^4-Y+Z(Y^2-Y)$
&
$4q^2-3q$
&
Decagonal
\\
\hline
\end{tabular}
\end{center}
Thus these point counts give the prime-power-indexed subsequences, with $q=p^n,$ of the square, hexagonal, octagonal, and decagonal numbers,
respectively.
\end{theorem}
\begin{proof}
By Corollary~\ref{cor:double-total-count}, we have
$$
|S_{d_1,d_2}(\Fq)|=q\bigl(q-1+N_0'\bigr),
$$
where
$N_0'=\#\{(Y,Z)\in\Fq^2\mid P(0,Z)=0,\ Q(0,Y,Z)=0\}.$
Thus it is enough to compute $N_0'$ in each case.

\begin{enumerate}
\item Suppose that
$P(0,Z)=Z^2$and $Q(0,Y,Z)=Y^2.$ Then $P(0,Z)=0$ gives $Z=0,$ and $Q(0,Y,0)=0$ gives $Y=0.$ Hence
$N_0'=1.$ Therefore,
$$
|S_{d_1,d_2}(\Fq)|
=
q(q-1+1)
=
q^2.
$$

\item Suppose that
$P(0,Z)=Z^q-Z$ and $ Q(0,Y,Z)=Y^2.$ Since every element of $\Fq$ is a root of $Z^q-Z,$ there are $q$ choices for $Z.$ For each such $Z,$ the equation $Q(0,Y,Z)=0$ gives $Y=0.$ Hence $N_0'=q.$
Therefore,
$$
|S_{d_1,d_2}(\Fq)|
=
q(q-1+q)
=
2q^2-q.
$$

\item Suppose that
$P(0,Z)=Z^q-Z$ and $ Q(0,Y,Z)=Y^2-ZY.$ Again $P(0,Z)=0$ gives all $Z\in\Fq.$ For a fixed $Z\in\Fq,$ we have $ Q(0,Y,Z)=Y^2-ZY=Y(Y-Z).$
If $Z=0,$ then $Q(0,Y,0)=Y^2,$ so there is exactly one root, namely $\be=0.$ If $Z\neq 0,$ then the two roots are $Y=0$ and $Y=Z.$ Hence
$ N_0'=1+2(q-1)=2q-1.$
Therefore,
$$
|S_{d_1,d_2}(\Fq)|
=
q(q-1+2q-1)
=
3q^2-2q.
$$

\item Suppose that $P(0,Z)=Z^q-Z$ and $Q(0,Y,Z)=Y^4-Y+Z(Y^2-Y).$ Since $P(0,Z)=Z^q-Z,$ every element $Z\in\Fq$ satisfies $P(0,Z)=0.$ We factor
$$
Q(0,Y,Z)
=
Y^4-Y+Z(Y^2-Y)
=
Y(Y-1)(Y^2+Y+Z+1).
$$
The factors $Y$ and $Y-1$ give the roots $Y=0$ and $Y=1$ for every
$Z\in\Fq.$ For the remaining factor, the equation $Y^2+Y+Z+1=0$ determines $Z=-Y^2-Y-1.$ Thus, for every $Y\in\Fq\setminus\{0,1\},$ there is exactly one corresponding $Z\in\Fq.$ Hence the total number of solutions is $ N_0'=2q+(q-2)=3q-2.$
Therefore,
$$
|S_{d_1,d_2}(\Fq)|
=
q(q-1+3q-2)
=
4q^2-3q.
$$

\end{enumerate}
This proves the stated point counts.
\end{proof}

\medskip
\noindent
\medskip
\noindent
A natural next direction is to find sequences of polynomials $P$ and $Q$ such that the reduced count $N_0'$ has a prescribed linear form in $q.$ 

\section*{Acknowledgements}
\noindent
The first author initiated this work while supported under the grant CRG/2022/007047 and later received financial support from the Department of Science and Technology, Government of India, under the WISE Post-Doctoral Fellowship Programme (Reference No.~DST/WISE-PDF/PM-114/2024(G)). The second author carried out part of this work at IIT Gandhinagar, first as a research intern under SRIP 2025, then as an M.Sc. student and subsequently as a Sabarmati Bridge Fellow. The third author was also supported by the grant CRG/2022/007047. All authors gratefully acknowledge the encouragement and support received from IIT Gandhinagar. AI assistance was used to generate images in this article and for drafting and language editing.


\end{document}